\documentclass[12pt,a4paper,oneside]{amsart}
\pdfoutput=1

\usepackage{amsmath} % Lots of maths functionality
\usepackage{amssymb} % Maths symbols
\usepackage{amsthm} % Maths environments: \begin{proof}, etc.
\usepackage[english]{babel} % Language and hyphenation.
\usepackage[font=small,justification=centering]{caption} % More flexibility for captioning figures
\usepackage[nodayofweek]{datetime}
\usepackage[shortlabels]{enumitem} % Change enumeration labelling with \begin{enumerate}[a)] etc.
\usepackage[a4paper]{geometry} % Changing page shape
\usepackage{ifthen} % For Dani's \begin{com} environment and \numberedtheorem
\usepackage{mathabx} % Contains \pnest symbol and more. Clashes with accents for \ring
\usepackage{mathtools} % Uses amsmath, fixes quirks and adds functionality.
\usepackage[dvipsnames]{xcolor} % Allow links to have colour. Needs to be before hyperref and tikz-cd
\usepackage{setspace} % Allows \onehalfspacing etc. for changing gaps between lines
\usepackage{tikz-cd} % Commutative diagrams
\usepackage{xfrac} % Nicer quotients 

\usepackage{amscd}
\usepackage{cleveref}
\usepackage[all]{xy}

\usepackage{float}
\restylefloat{table}

\usepackage{array}
\newcolumntype{P}[1]{>{\centering\arraybackslash}p{#1}}

\makeatletter
\def\@tocline#1#2#3#4#5#6#7{\relax
  \ifnum #1>\c@tocdepth % then omit
  \else
    \par \addpenalty\@secpenalty\addvspace{#2}%
    \begingroup \hyphenpenalty\@M
    \@ifempty{#4}{%
      \@tempdima\csname r@tocindent\number#1\endcsname\relax
    }{%
      \@tempdima#4\relax
    }%
    \parindent\z@ \leftskip#3\relax \advance\leftskip\@tempdima\relax
    \rightskip\@pnumwidth plus4em \parfillskip-\@pnumwidth
    #5\leavevmode\hskip-\@tempdima
      \ifcase #1
       \or\or \hskip 1em \or \hskip 2em \else \hskip 3em \fi%
      #6\nobreak\relax
    \dotfill\hbox to\@pnumwidth{\@tocpagenum{#7}}\par% <---- \dotfill -> \hfill
    \nobreak
    \endgroup
  \fi}
\makeatother

\newtheorem{thm}{Theorem}[section]
\newtheorem{theorem}[thm]{Theorem} \newtheorem{proposition}[thm]{Proposition} %added by EMP
\newtheorem{lemma}[thm]{Lemma}
\newtheorem{corollary}[thm]{Corollary}

\theoremstyle{definition}

\newtheorem{definition}[thm]{Definition}
\newtheorem{remark}[thm]{Remark}
\newtheorem{question}[thm]{Question}

\newtheorem{notation}[thm]{Notation}

\newtheorem{example}{Example}

\DeclareMathOperator{\Hom}{\mathrm{Hom}}

\newcommand{\dbN}{\mathbb{N}}

\newcommand{\dbZ}{\mathbb{Z}}

\newcommand{\calP}{{\mathcal P}}
\newcommand{\calQ}{{\mathcal Q}}

\newcommand{\hocolim}{\mathop{\rm hocolim}}
\newcommand{\hofib}{{\rm hofib}}

\newcommand{\Ab}[1]{\mathsf{Ab}(#1)}
\newcommand{\AbCo}[1]{\mathsf{AbCo}(#1)}
\newcommand{\mAb}[1]{\mathsf{mAb}(#1)}
\newcommand{\mAbCo}[1]{\mathsf{mAbCo}(#1)}

\newcommand{\isom}{\mathrm{Isom}}

\usepackage{listings}
\usepackage{tikz}
\usetikzlibrary{arrows,quotes}
\tikzstyle{blackNode}=[fill=black, draw=black, shape=circle]
\usetikzlibrary{decorations.pathmorphing}
\tikzset{snake it/.style={decorate, decoration=snake}}
\usetikzlibrary{calc}

\usepackage{centernot}

\newcommand{\Z}{\mathbb{Z}}

\newcommand{\mbr}{\mathbb{R}}

\newcommand{\mbe}{\mathbb{E}}

\newcommand{\mbs}{\mathbb{S}}

\newcommand{\pslt}{\widetilde{\mathrm{PSL}}_2(\mbr)}

\newcommand{\hyp}{\mathbb{H}^2}
\newcommand{\hypp}{\mathbb{H}^3}

\newcommand{\nil}{\mathrm{Nil}}
\newcommand{\sol}{\mathrm{Sol}}

\newcommand{\hxr}{\mathbb{H}^2\times\mathbb{E}}

\usetikzlibrary{external}
\usepackage[style=alphabetic, citestyle=alphabetic, sorting=nyvt]{biblatex}
\usepackage{csquotes}

\usepackage{tkz-euclide} 
\usetikzlibrary{matrix}
\newcounter{commentcounter}

\makeatletter
\let\@wraptoccontribs\wraptoccontribs
\makeatother

\newcommand{\Ecom}[1]{E_{\mathsf{com}}(#1)}
\newcommand{\Bcom}[1]{B_{\mathsf{com}}(#1)}
\newcommand{\pt}{\mathrm{pt}}

\title[]{The classifying space for commutativity of geometric orientable 3-manifold groups}

\dedicatory{To Alejandro Adem on the occasion of his 60th birthday.}

\author{Omar Antol\'{\i}n-Camarena}
\address{Omar Antol\'{\i}n-Camarena, Universidad Nacional Auton\'onoma de M\'exico}
\email{omar@matem.unam.mx}
\author{Luis Eduardo Garc\'{\i}a-Hern\'andez}
\address{Luis Eduardo Garc\'{\i}a-Hern\'andez, Universidad Nacional Auton\'onoma de M\'exico}
\email{legh@ciencias.unam.mx}
\author{Luis Jorge S\'anchez Salda\~na}
\address{Luis Jorge S\'anchez Salda\~na, Universidad Nacional Aut\'onoma de M\'exico}
\email{luisjorge@ciencias.unam.mx}
\date{\today}

\subjclass{}

\begin{document}

\maketitle
\begin{abstract} For a topological group $G$ let $\Ecom G$ be the total space of the universal transitionally commutative principal $G$-bundle as defined by Adem--Cohen--Torres-Giese. So far this space has been most studied in the case of compact Lie groups; but in this paper we focus on the case of infinite discrete groups.

For a discrete group $G$, the space $\Ecom G$ is homotopy equivalent to the geometric realization of the order complex of the poset of cosets of abelian subgroups of $G$. We show that for fundamental groups of closed orientable geometric $3$-manifolds, this space is always homotopy equivalent to a wedge of circles. On our way to prove this result we also establish some structural results on the homotopy type of $\Ecom{G}$.
\end{abstract}

%\tableofcontents

\section{Introduction}

For a topological group $G$ one can topologize the set of homomorphisms $\Hom(\mathbb{Z}^n, G)$ as the subspace of $G^n$ consisting of $n$-tuples of pairwise commuting elements of $G$. These spaces of commuting tuples form a simplicial subspace of the usual simplicial space model for the classifying space $BG$. The geometric realization of this simplicial subspace is called the classifying space for commutativity of $G$ and denoted by $\Bcom{G}$. It was first introduced in \cite{ACTG} and further studied in \cite{AG}, where a notion of transitionally commutative principal $G$-bundle was introduced and shown to be classified by $\Bcom{G}$. The total space of the universal transitionally commutative principal $G$-bundle is denoted $\Ecom{G}$. When $G$ is abelian, $\Bcom{G}$ agrees with $BG$ and $\Ecom{G}$ agrees with $EG$. One can consider how far $\Ecom{G}$ is from being contractible as a sort of measure of non-abelianness of $G$, and indeed, it has been shown \cite{ACGV} that for compact, not necessarily connected, Lie groups $G$, one has that $G$ is abelian if and only if $\Ecom{G}$ is contractible, and in fact, if and only if $\pi_k(\Ecom{G}) = 0$ for $i \in \{1,2,4\}$ (note the absence of $3$ in that set!).

It is fair to say that most of the attention given to $\Bcom{G}$ and $\Ecom{G}$ so far has been in the case of Lie groups, and more particularly compact Lie groups, but the definitions are also interesting for discrete groups. In the case of a discrete group $G$, the theorem cited above can improved to say that $G$ is abelian if and only if $\Ecom{G}$ is simply-connected, as first shown in \cite{Okay2014} and with a different proof in \cite{ACGV}. For discrete groups one can give a simple combinatorial model of the homotopy type of $\Ecom{G}$ as a poset, namely, $\Ecom{G}$ is homotopy equivalent to the geometric realization of the order complex of the poset of cosets of abelian subgroups of $G$.

In this paper we use this poset description to compute the homotopy type of $\Ecom{G}$ for fundamental groups of orientable geometric $3$-manifolds and show that they are all wedges of countably many circles. Recall that Thurston
showed that there are eight $3$-dimensional maximal geometries up to
equivalence (\cite[Theorem 5.1]{Sc83}): $\mbs^3$, $\mbe^3$, $\hypp$,
$\mbs^2\times\mbe$, $\hxr$, $\pslt$, $\nil$, and $\sol$. The main result of this paper is the following:

\begin{theorem}\label{thm:main}
    Let $G$ be the fundamental group of an orientable geometric 3-manifold. Then $\Ecom{G}$ is homotopically equivalent to $\bigvee_{I} S^1$, where $I$ is a (possibly empty) countable index set. Moreover, we have the following:
    \begin{enumerate}
        \item $I$ is empty if and only if $G$ is abelian, in which case $\Ecom{G}$ is contractible.
        \item $I$ is infinite if and only if $G$ is infinite and nonabelian.
        \end{enumerate}
\end{theorem}

From \Cref{thm:main} we conclude that $\Ecom{G}$ is a finite nontrivial wedge of circles exactly when $G$ is finite and nonabelian which, by Perelman's theorem, happens when $G$ is the fundamental group of a spherical manifold. In that case we explicitly compute the number of circles in the wedge decomposition, see Section 5. Some of these calculations were done using GAP.

As a consequence of \Cref{thm:main}, \Cref{thm:direct:free:product} and the Kneser--Milnor prime decomposition theorem (see \Cref{prime decomposition}), we obtain the following theorem:

\begin{theorem}\label{thm:main:prime:decomposition}
    Let $M$ be a 3-manifold with fundamental group $G$. Assume that $P_1 \# \cdots \# P_n$ is the prime decomposition of $M$, with $n\geq 2$, and each $P_i$ is a geometric 3-manifold. Then   $\Ecom{G}$ is homotopically equivalent to $\bigvee_{\dbN} S^1$.
\end{theorem}

The proof of \Cref{thm:main} is done case by case, that is, analyzing the possible fundamental groups that appear for each of the eight 3-dimensional geometries. In the following table we summarize the references for each geometry:

\begin{table}[H]
\centering
\begin{tabular}{|P{5.5cm}|P{7cm}|}
 \hline
 \textbf{Type of geometry}  &\textbf{Worked out in} \\
 \hline
   $\mbs^3$    & \Cref{thm:Ecom:Q4n},\Cref{thm:Ecom:p'},\Cref{thm:Ecom:P48:P120} and \Cref{thm:Ecom:FiniteDihedral}  \\
  \hline 
 $\hypp$  & \Cref{thm:Ecom:hyperbolic}  \\
 \hline
  $\mbs^2\times\mbe$  & \Cref{sec:S2xE}  \\
 \hline
 $\hxr$ or $\pslt$  & \Cref{thm:Ecom:H2xE}  \\
 \hline
  $\mbe^3$   & \Cref{thm:Ecom:G1-G5} and \Cref{thm:Ecom:G6}  \\
  \hline
    $\nil$ &  \Cref{thm:nil:semidirect}, \Cref{thm:nil:orientable:orbifold}, \Cref{thm:Ecom:nil:pg} and \Cref{thm:case:pgg}  \\
  \hline
   $\sol$ &  \Cref{prop:sol:semidirect} and \Cref{thm:Ecom:Sol:amalgam}  \\
 \hline
\end{tabular}
\caption{   $\Ecom{G}$ of the fundamental group $G$ of an orientable geometric 3-manifold}
\label{table:summary:Ecom}
\end{table}

Let us say something about the strategy of the proof of the theorems  in \Cref{table:summary:Ecom}. The space $\Ecom{G}$ is homotopically equivalent to the geometric realization of the poset $\mAbCo{G}$ that consists of cosets of maximal abelian subgroups of $G$ and their intersections, ordered by inclusion. For most of the fundamental groups we dealt with, it turned out that $\mAbCo{G}$ has height 1, that is, its geometric realization has dimension one, thus it is straightforward that has the homotopy type of a wedge of circles. The exceptional cases where given by amalgamated products of the form $K\ast_{\dbZ^2}K$, where $K$ is the fundamental group of the Klein bottle. Such groups appear in geometries $\mbe^3$, $\nil$, and $\sol$. In these cases $\mAbCo{G}$ has height 2, and therefore we have triangles in its geometric realization. Nevertheless, it is always the case that some triangles have free faces, that is they have a face that does not belong to any other triangle, and they can be deformation retracted to the other two faces. Finally, after collapsing these triangles it happens that all the remaining triangles have a free face and they also can be collapsed. In conclusion, via a two-staged collapsing procedure we see that the geometric realization of  $\mAbCo{G}$ can be deformation retracted to its 1-skeleton, and this concludes the proof.

\subsection*{Some byproduct results}
On our way to prove \Cref{thm:main}, we establish some structural results on the homotopy type of $\Ecom{G}$ and we also computed the homotopy type of $\Ecom{G}$ of some other groups. Next, we state these results for the sake of the reader. The first theorem tell us that the homotopy type of $\Ecom{G}$ behaves well with respect to free and direct products.

\begin{theorem}[\Cref{thm:direct:free:product}]\label{thm:direct:free:product:intro} Let $G_1$ and $G_2$ be nontrivial groups. Then
\begin{enumerate}
    \item $\Ecom{G_1 \ast G_2} \simeq \bigvee_{\mathbb{N}} \Ecom{G_1} \vee \bigvee_{\mathbb{N}} \Ecom{G_2} \vee \bigvee_{\mathbb{N}} S^1$, and
    \item $\Ecom{G_1 \times G_2} \simeq \Ecom{G_1} \times \Ecom{G_2}$.
\end{enumerate}
\end{theorem}

Fundamental groups of closed hyperbolic 3-manifolds are very particular cases of hypebolic groups (in the sense of Gromov). We proved the following.

\begin{theorem}[\Cref{thm:torsionfree:hyperbolic}]\label{thm:torsionfree:hyperbolic:intro}
Let $G$ be a torsion-free hyperbolic group in the sense of Gromov. Assume $G$ is not virtually cyclic. Then
\[\Ecom{G} \simeq \bigvee_{\mathbb{N}} S^1.\]
\end{theorem}

In the following theorem there is a computation of  the homotopy type of $\Ecom{G}$ for five of the nonabelian wallpaper groups, i.e. 2-dimensional crystallographic groups.

\begin{theorem}[\Cref{lemma:mabco:klein:bottle} and \Cref{prop:2dim:crystallographic}]\label{thm:wallpaper}
    Let $G$ be one of the following wallpaper groups: 
    \begin{enumerate}
    \item the fundamental group of the Klein bottle $\dbZ\rtimes_{-1}\dbZ$,
    \item $\dbZ\rtimes_{A_n} \dbZ/n$, where $A_n$ is an integral matrix of order 2, 3, 4 or 5 (see \Cref{eq:matrices:2crystallographic} for a concrete description).
    \end{enumerate}
Then $\Ecom{G}\simeq \bigvee_{\mathbb{N}} S^1$.
\end{theorem}

\subsection*{Open questions and further research}

Note that the center $Z$ of a group $G$ is exactly the intersection of all maximal abelian subgroups of $G$. Provided that $Z$ has finite index in $G$ we conclude that $\mAbCo{G}$ is a finite poset, and as a conclusion $\Ecom{G}$ has the homotopy type of a finte $CW$-complex. This conclusion does not apply to groups that contain an abelian subgroup of finite index like the fundamental group of the Klein bottle, see  \Cref{lemma:mabco:klein:bottle}. This leads to the following question.

\begin{question}
    Let $G$ be a group. Assume that $\Ecom{G}$ has the homotopy type of a finite CW-complex. Is it true that $G$ contains its center as a finite index subgroup?
\end{question}

It is well-known that there are 17 wallpaper groups. The only abelian of these groups is $\dbZ^2$ for which $E_{\mathsf{com}}$ is contractible. On the other hand, for the groups listed in \Cref{thm:wallpaper}, we proved that $\Ecom{G}$ has the homotopy type of an infinite countable wedge of circles. It is natural to ask for the analogous computation for the remaining wallpaper groups.

\begin{question}
    Let $G$ be a wallpaper group. What is the homotopy type of $\Ecom{G}$? Is it true that $\Ecom{G}\simeq \bigvee_{\mathbb{N}} S^1$?
\end{question}

Let a closed, prime, oriented 3-manifold $M$ that is not homeomorphic to $S\times D^2$, $T^2\times I$ nor the nontrivial $I$-bundle over the Klein bottle. Assume $G$ is not geometric. Then there exists a nonempty collection
$T\subseteq M$ of disjoint incompressible tori (i.e. two sided properly embedded and $\pi_1$-injective), such that each component of
$M - T$ is geometric, see for instance \cite[Theorem 1.9.1]{AFW15}. A deep analysis of this decomposition and the results in the present article, will lead to a classification of the (maximal) abelian subgroups of the fundamental group $G$ of $M$. A natural question that is left by \Cref{thm:main} is the following.

\begin{question}
    Let $G$ be the fundamental group of a prime non-geometric 3-manifold. What is the homotopy type of $\Ecom{G}$?
\end{question}

\subsection*{Outline of the paper}
In Section 2 we establish that for any discrete group $G$ the space $\Ecom{G}$ is homotopically equivalent to the geometric realization of $\mAbCo{G}$, which is the corner stone of all of our computations. Section 3 is devoted to recall what is needed, for our purpose, about the theory of 3-manifolds. In Section 4 we prove some structural results on the homotopy type of $\Ecom{E}$ that will be useful in the proof of \Cref{thm:main}; among other things we prove in this section \Cref{thm:direct:free:product:intro}. In Sections 5 to 8 we compute the homotopy type of $\Ecom{G}$ for geometries $\mbs^3$, $\hypp$, $\mbs^2\times\mbe$, and $\hxr$ respectively. The main goals of Sections 9 and 10 are to set up notation and preliminary results to deal with groups of the form $K\ast_{\dbZ^2}K$, with $K$ the fundamental group of the Klein bottle; is in this section that we compute also the homotopy type of $\Ecom{G}$ for some wallpaper groups. Finally in Sections 11 to 13 we compute the homotopy type of  $\Ecom{G}$ for geometries $\mbe^3$, $\nil$ and $\sol$.

\vskip 10pt

\emph{Acknowledgments} O.A-C. and L.E.G.-H. gratefully acknowledge support from CONACyT Ciencia de Frontera 2019 grant CF217392. L.J.S.S is grateful for the financial support of DGAPA-UNAM grant PAPIIT IA106923.

\section{Posets of cosets of abelian subgroups}

In this section we establish that for every discrete group $G$, the space $\Ecom{G}$ is homotopically equivalent to the geometric realization of the poset $\mAbCo{G}$ of all cosets of maximal abelian subgroups of $G$ and their intersections, see \Cref{Ecom_mAbCo}. This result is the corner stone of all of our computations.

\begin{definition}[Height of a poset]
Let $P$ be a poset. We say a chain $x_0< x_1 < \cdots < x_n$ in $P$ has length $n$, and the \emph{height of $P$} is the greatest length of a chain in $P$. Note that the height of $P$ coincides with the dimension of the geometric realization (of the nerve)  of $P$.  
\end{definition}

\begin{definition}[Posets of abelian subgroups]
Let $G$ be a group. Define the following posets
\begin{itemize}
\item $\Ab{G}$ as the poset of all abelian subgroups of $G$.
\item $\AbCo{G}$ as the poset of all cosets of abelian subgroups of $G$.
\item $\mAbCo{G}$ as the poset of all cosets $gB$ such that $g\in G$ and $B$ is an intersection of maximal abelian subgroups of $G$. 
\end{itemize}
In all of these posets the order relation is the one given by inclusion.
\end{definition}

\begin{remark}\label{remark:maximal:contains:center}
Let $G$ be a group and $Z$ its center.  Notice that all maximal abelian subgroups of $G$ contain $Z$. In fact, one can show that $Z$ is equal to the intersection of all maximal abelian subgroups of $G$.
\end{remark}

All of our computations on the homotopy type of $\Ecom{G}$ rely on the following result, which will be used from now on without further mention. Therefore all throughout the article we will only deal with (the geometric realization of)  $\mAbCo{G}$ and $\AbCo{G}$.

\begin{proposition}\label{Ecom_mAbCo}
Given a discrete group $G$, the following CW-complexes are homotopy equivalent.
\begin{enumerate}
    \item $\Ecom{G}$,
    \item the geometric realization of $\AbCo{G}$.
    \item the geometric realization of $\mAbCo{G}$, 
\end{enumerate}
\end{proposition}

\begin{proof}
This result is essentially contained in \cite{Okay2014}, but for the reader's convenience we sketch the argument in addition to providing references. Note that \cite{Okay2014} states these results only for finite groups, but the argument only requires discreteness.

Recall that $\Ecom{G}$ is the homotopy fiber of the canonical map $\Bcom{G} \to BG$, and that $\Bcom{G}$ in turn is homotopy equivalent to the $\hocolim_{A \in \Ab{G}} BA$. Since homotopy colimits are preserved by homotopy pullback along a fixed map, we can pullback the homotopy colimit for $\Bcom{G}$ along the map $\pt\to BG$ to obtain (\cite[Equation (3.5.1)]{Okay2014}):
\[\Ecom{G} \simeq \hocolim_{A \in \Ab{G}} \hofib{(BA \to BG)} \simeq \hocolim_{A \in \Ab{G}} G/A.\]

By Thomason's theorem the homotopy colimit of a functor like $A \to G/A$ whose values are discrete spaces can be computed as the geometric realization of the nerve of the Grothendieck construction (or category of elements) of that functor. In this case the objects of this Grothendieck construction would be the cosets of abelian subgroups of $G$ and unwinding the definition of the morphisms they turn out to be simply inclusions of cosets (see the remark immediately following \cite[Theorem 6.1]{Okay2014}). This establishes the equivalence of (1) and (2). The equivalence of (2) and (3) is simply because $\mAbCo{G}$ is cofinal in $\AbCo{G}$.
\end{proof}

The proof of the following lemma is elementary and it is left to the reader. We will also be using this result all throughout the paper without mentioning it. 

\begin{lemma}
Let $G$ be a group. Then (the geometric realization of) $\mAbCo{G}$ is connected. 
\end{lemma}

\section{ Geometric 3-manifolds}\label{sec-3mflds}

In this section we will review a bit of $3$-manifold theory. For more details see \cite{Sc83}, \cite{Mo05}.

\subsection{Geometric 3-manifolds}
A \emph{Riemannian} manifold $X$ is a smooth
manifold that admits a Riemannian metric.  If the isometry group $\isom(X)$
acts transitively, we say $X$ is \emph{homogeneous}.  If in addition $X$ has a quotient of finite
volume, $X$ is \emph{unimodular}.  A \emph{geometry} is a simply-connected,
homogeneous, unimodular Riemannian manifold along with its isometry group.  Two
geometries $(X,\isom(X))$ and $(X',\isom(X'))$ are \emph{equivalent} if
$\isom(X)\cong\isom(X')$ and there exists a diffeomorphism $X\to X'$ that
respects the $\isom(X), \isom(X')$ actions.  A geometry $(X,\isom(X))$ (often
abbreviated $X$) is \emph{maximal} if there is no Riemannian metric on $X$ with
respect to which the isometry group strictly contains $\isom(X)$.  A manifold
$M$ is called \emph{geometric} if there is a geometry $X$ and discrete subgroup
$\Gamma\leq\isom(X)$ with free $\Gamma$-action on $X$ such that $M$ is
diffeomorphic to the quotient $X/\Gamma$; we also say that $M$ \emph{admits a
geometric structure} modeled on $X$.  Similarly, a manifold with non-empty
boundary is geometric if its interior is geometric.

It is a consequence of the uniformization theorem that compact surfaces
(2-manifolds) admit Riemannian metrics with constant curvature; that is, compact
surfaces admit geometric structures modeled on $\mbs^2$, $\mbe^2$, or $\hyp$. 
In dimension three, we are not guaranteed constant curvature.  Thurston
showed that there are eight $3$-dimensional maximal geometries up to
equivalence (\cite[Theorem 5.1]{Sc83}): $\mbs^3$, $\mbe^3$, $\hypp$,
$\mbs^2\times\mbe$, $\hxr$, $\pslt$, $\nil$, and $\sol$.

\subsection{Prime decomposition}
The material of this subsection it is only stated by the sake of completeness as these concepts are mentioned in the statement of \Cref{thm:main:prime:decomposition}. The prime decomposition theorem will not be used anywhere else in this article.

A \emph{closed} $n$-manifold is an $n$-manifold that is compact with empty
boundary.  A \emph{connected sum} of two $n$-manifolds $M$ and $N$, denoted $M\#
N$, is a manifold created by removing the interiors of a smooth $n$-disc $D^n$
from each manifold, then identifying the boundaries $\mbs^{n-1}$.  An $n$-manifold is
\emph{nontrivial} if it is not homeomorphic to $\mbs^n$.  A \emph{prime}
$n$-manifold is a nontrivial manifold that cannot be decomposed as a connected
sum of two nontrivial $n$-manifolds; that is, $M=N\# P$ for some $n$-manifolds
$N,P$ forces either $N=\mbs^n$ or $P=\mbs^n$.  An $n$-manifold $M$ is called
\emph{irreducible} if every 2-sphere $\mbs^2\subset M$ bounds a ball $D^3\subset M$. It is well-known that all orientable prime manifolds are irreducible with the exception of $S^1\times
S^2$. The following is a well-known theorem of Kneser (existence) and Milnor (uniqueness).

\begin{theorem}[Prime decomposition]
\label{prime decomposition}
Let $M$ be a closed oriented nontrivial 3-manifold.  Then $M=P_1\#\cdots\#
P_n$ where each $P_i$ is prime.  Furthermore, this decomposition is unique up to
order and homeomorphism.
\end{theorem}

%Another well known result we will need is the Jaco--Shalen--Johannson decomposition. The version stated here also uses Perelman's work.

%\begin{theorem}[JSJ decomposition]
%\label{jsj decomposition}
%For a closed, prime, oriented 3-manifold $M$ there exists a collection
%$T\subseteq M$ of disjoint incompressible tori (i.e. two sided properly embedded and $\pi_1$-injective), such that each component of
%$M - T$ is either a hyperbolic or a Seifert fibered manifold.  A minimal
%such collection $T$ is unique up to isotopy.
%\end{theorem}

\section{Some structural results on the homotopy type of $\Ecom{G}$}

In this section we prove some foundational results on the homotopy type of $\Ecom{G}$, which are interesting in their own right. These results will be used in the rest of the paper.

\begin{proposition}\label{thm:central:extension}
Consider the following extension of groups
\[1\to K \to G \to Q \to 1.\]
Let $\calP$ be the poset of cosets of all abelian subgroups of $G$ that contain $K$, and let $\calQ$ be the poset of cosets of subgroups of $Q$ whose preimage in $G$ is abelian. Then $\calP$ and $\calQ$ are isomorphic.
Moreover, if $\calP_{\mathsf{max}}$ (resp. $\calQ_{\mathsf{max}}$) is the poset of intersections of finitely many maximal members of $\calP$ (resp. $\calQ$), then $\calP_{\mathsf{max}}$ is isomorphic to $\calQ_{\mathsf{max}}$.
\end{proposition}
\begin{proof}
The standard bijection between subgroups of $G$ containing $K$ and subgroups of $Q$ restricts to a bijection between abelian subgroups of $G$ containing $K$ and subgroups of $Q$ with abelian preimage in $G$. Given a subgroup $H$ of $G$ that contains $K$ let $\bar{H}$ be its image in $Q$. Notice that $[G:H] = [Q : \bar{H}]$. This implies that the bijection we have between subgroups also gives rise to a isomorphism of poset cosets.
\end{proof}

We will often apply this proposition in cases where every abelian subgroup of $Q$ has abelian preimage in $G$, so that $\calQ$ and $\calQ_{\mathsf{max}}$ are actually $\mAb{Q}$ and $\mAbCo{Q}$ respectivelly. Note further in that case the extension is central.

\begin{theorem}\label{thm:countable:wedge}
Let $G$ be a nonabelian countably infinite group. Assume that for every pair of distinct maximal abelian subgroups $A$ and $B$ of $G$, $A\cap B=1$. Then, $\Ecom{G}\simeq \bigvee_{\mathbb{N}} S^1$.
\end{theorem}
\begin{proof}
The hypothesis implies directly that $\mAbCo{G}$ is of height 1, thus its geometric realization has dimension 1 and it has a finite or countable number of cells. We conclude, by \Cref{Ecom_mAbCo} that $\Ecom{G}$ is homotopically equivalent to a wedge of circles. It only remains to be proved that there are infinitely many of these circles.

Consider $Y$ the subcomplex of the geometric realization $X$ of $\mAbCo{G}$ given as follows: 

\begin{itemize}
    \item All the vertices of $X$ belong to $Y$.
    \item All edges of the form $\{\{g\},A\}$  with $g\in A$ and $A$ a maximal abelian subgroup of $G$, belong to $Y$. 
    \item For each vertex of $X$ of the form $xA$ with $A\neq 1$, choose and fix a representative of the corresponding coset, lets say $x$. The edge $\{\{x\}, xA\}$ belongs to $Y$.
\end{itemize}

We claim $Y$ is connected. First, any coset of the trivial group, say $\{g\}$, is contained in some maximal abelian subgroup $A$, and $Y$ has edges $\{\{1\},A\},\{\{g\},A\}$, which connect $\{g\}$ to $\{1\}$. Second, any coset of a maximal abelian subgroup $A$, has a chosen representative $x$ and $Y$ has the edge $\{\{x\}, xA\}$ connecting $xA$ to a coset of the trivial group, which we already connected to $\{1\}$.

Since $Y$ is connected and contains all vertices of $X$, a maximal subtree $T$ of $Y$ is also a maximal subtree of $X$ containing all vertices. To finish the proof it is enough to show that there are infinitely many edges in $X$ that do not belong to $Y$. Among the maximal abelian subgroups of $G$ there is either a subgroup $H$ of infinite index or an infinite proper subgroup $K$; we deal with each case separately. Let $H$ be as in the first case, note that, by construction of $Y$, for each coset $gH\neq H$ there is exactly one edge in $Y$  of the form $\{\{x\},gH\}$, and since $H$ is nontrivial there is an element $y\neq x$ such that the edge $\{\{y\},gH\}$ does not belong to $Y$; this finishes the proof in the first case. Let $K$ be as in the second case, since $K$ is proper there is a coset $gK\neq K$, as in the previous case, all but one of the edges of the form $\{\{y\},gK\}$ do not belong to $Y$; this finishes the proof of the second case.
\end{proof}

The following corollary is a direct consequence of \Cref{thm:central:extension} and \Cref{thm:countable:wedge}.

\begin{corollary}\label{cor:height:1}
Let $G$ be a group, and denote by $Z$ the center of $G$. Assume that for every pair of distinct maximal abelian subgroups $A$ and $B$ of $G$, $A\cap B=Z$, and assume $Z$ has countable infinite index in $G$. Then, $\Ecom{G}\simeq \bigvee_{\mathbb{N}} S^1$.
\end{corollary}

\begin{theorem}\label{thm:direct:free:product} Let $G_1$ and $G_2$ be nontrivial groups. Then
\begin{enumerate}
    \item $\Ecom{G_1 \ast G_2} \simeq \bigvee_{\mathbb{N}} \Ecom{G_1} \vee \bigvee_{\mathbb{N}} \Ecom{G_2} \vee \bigvee_{\mathbb{N}} S^1$, and
    \item $\Ecom{G_1 \times G_2} \simeq \Ecom{G_1} \times \Ecom{G_2}$.
\end{enumerate}
\end{theorem}

\begin{proof}
\begin{enumerate}
    \item By Kurosh's theorem, any abelian subgroup of $G_1 \ast G_2$ is either a subgroup of a conjugate of $G_1$, a subgroup of a conjugate of $G_2$, or an infinite cyclic subgroup generated by an element which is not conjugate to any element of $G_1$ or $G_2$. Furthermore, abelian subgroups of different types intersect only in the trivial subgroup. This means the poset of abelian (resp. maximal abelian) subgroups of $G_1 \ast G_2$ is a wedge of the all the abelian (resp. maximal abelian) subgroups of all the conjugates of $G_1$, all the abelian (resp. maximal abelian) subgroups of all the conjugates of $G_2$ and all the cyclic subgroups generated by elements (resp. primitive elements) not conjugate to $G_1$ or $G_2$. For the coset posets, we have
    \[\hspace{1cm}\mAbCo{G}=\bigcup_{y\in G/G_1}\bigsqcup_{g\in G/G_1^y}g\cdot\mAbCo{G_1^y}\cup
    \bigcup_{x\in G/G_2}\bigsqcup_{g\in G/G_2^x}g\cdot\mAbCo{G_2^x}\cup
    \mathsf{Co}P\]
    where $P$ is the poset consisting of the subgroup 1 and all maximal abelian subgroups generated by a primitive element not conjugate to either factor; and $\mathsf{Co} P$ denotes the poset of cosets of subgroups in $P$.
    
    Consider the subcomplex $Y$ of the geometric realization $X$ of $\mAbCo{G}$ given as follows: 

\begin{itemize}
    \item For each $g\in G - \{1\}$ choose and fix a maximal abelian subgroup $M_g$ of $G$, such that $g\in M_g$. The edges $\{\{1\},M_g\}$ and $\{\{g\},M_g\}$ belong to $Y$.
    \item For each vertex of $X$ of the form $xA$ with $A$ a maximal abelian subgroup and $x A\neq A$, choose and fix a representative of the coset $xA$, say $x$. The edge $\{\{x\}, xA\}$ belongs to $Y$.
    \item All the vertices in the previous edges belong to $Y$, i.e., the vertices of $Y$ are the singletons and the cosets of all maximal abelian subgroups of $G$. 
\end{itemize}

We claim $Y$ is connected. First, any coset of the trivial group, say $\{g\}$, is contained in $M_g$, and $Y$ has edges $\{\{1\},M_g\},\{\{g\},M_g\}$, which connect $\{g\}$ to $\{1\}$. Second, any coset of a maximal abelian subgroup $A$, has a chosen representative $x$ and $Y$ has the edge $\{\{x\}, xA\}$ connecting $xA$ to the coset $\{x\}$ of the trivial group, which we already connected to $1$. Also from the construction of $Y$, this subcomplex is a tree; indeed, taking $\{1\}$ as the root, the succesive levels of the tree are the $M_g$, then the $\{g\}$ for $g\neq 1$, then the non-subgroup cosets $xA$. 

We have that $Y$ has non-empty intersection with every $g\cdot \mAbCo{G_i^y}$. Also for each $xH$ coset in $P$ with $xH\neq H$, there are infinitely many edges of $X$ that do not belong to $Y$. Therefore, after collapsing $Y$, we have a complex with homotopy type the wedge of infintely many $\Ecom{G_1}$ plus a number of copies of $S^1$, infintely many $\Ecom{G_2}$ plus a number of copies of $S^1$, and infinitely many copies of $S^1$ from $P$. 
    
    \item Any abelian subgroup $M$ of $G_1 \times G_2$ is contained in $\pi_1(M) \times \pi_2(M)$, which is also abelian. Therefore, the maximal abelian subgroups of $G_1\times G_2$ are of the form $A_1\times A_2$ with $A_1$ and $A_2$ maximal abelian subgroups of $G_1$ and $G_2$, respectively. This implies that the poset $\mAbCo{G_1\times G_2}$ is isomorphic to $\mAbCo{G_1}\times\mAbCo{G_2}$ with the coset $(g_1,g_2)A_1 \times A_2$ corresponding to the pair $(g_1 A_1, g_2 A_2)$. As the geometric realization of a poset respects products the result follows. 
\end{enumerate}
\end{proof}

\begin{theorem}\label{thm:torsionfree:hyperbolic}
Let $G$ be a torsion-free hyperbolic group in the sense of Gromov. Assume $G$ is not virtually cyclic. Then
\[\Ecom{G} \simeq \bigvee_{\mathbb{N}} S^1.\]
\end{theorem}
\begin{proof} It is well-known that in a hyperbolic group the centralizer of any infinite cyclic group is virtually cyclic, in particular, every abelian subgroup of $G$ is either trivial or infinite cyclic (see for instance \cite[Corollary~III.3.10]{BridsonHaefligerBook}). On the other hand, if $H$ and $K$ are infinite cyclic maximal subgroups of $G$, then either $H=K$ or $H\cap K=1$, (see for instance \cite[Remark~7]{JPL06}). Since $G$ is not virtually cyclic, and in particular countably infinite nonabelian,  now the conclusion follows from \Cref{thm:countable:wedge}.
\end{proof}

The following lemma will not be used in the proof of the main result, but we record it for future reference.

\begin{theorem}\label{thm:amalgam:of:abelians}
Let $A$ and $B$ be abelian groups with a common proper subgroup $F$. Then \[\Ecom{A \ast_F B} \simeq \bigvee_{\mathbb{N}} S^1.\]
\end{theorem}

\begin{proof}
Kurosh's theorem implies that any abelian subgroup of $(A/F) \ast (B/F)$ is either a subgroup of a conjugate of $A/F$, a subgroup of a conjugate of $B/F$ or infinite cyclic. In all three cases the pre-image of the subgroup in $A \ast_F B$ can be easily seen to be abelian. Thus, \Cref{thm:central:extension} implies that $\mAbCo{A \ast_F B}$ and $\mAbCo{(A/F) \ast (B/F)}$ are isomorphic. Finally, \Cref{thm:direct:free:product} establishes the claim, since $\Ecom{A/F}$ and $\Ecom{B/F}$ are contractible.
\end{proof}

\begin{example}
Let  $G=SL_2(\Z)$ and $H=PSL_2(\Z)$. Since $G\cong\Z/4\ast_{\Z/2} \Z/6$ and $H\cong\Z/2\ast \Z/3$, \Cref{thm:amalgam:of:abelians} implies that both $\Ecom{G}$ and $\Ecom{H}$ have the homotopy type of $\bigvee_{\mathbb{N}} S^1$.
\end{example}

\section{3-manifolds modeled on $\mbs^3$}
The fundamental group of a 3-manifolds modeled on $\mbs^3$ is either finite cyclic or a direct product of a finite cyclic group with a group described in the following list, see for instance \cite[p. 12]{AFW15}.

\begin{enumerate}
    \item $Q_{4n} = \langle x,y \mid x^2=(xy)^2=y^n \rangle$ where $n\geq 2$,
    \item $P_{48}=\langle x,y \mid x^2=(xy)^3=y^4, x^4=1 \rangle$, 
    \item $P_{120}=\langle x,y \mid x^2=(xy)^3=y^5, x^4=1 \rangle$,
    \item $D_{2^m(2n+1)}=\langle x,y \mid x^{2^m}=1, y^{2n+1}=1, xyx^{-1}=y^{-1} \rangle$, where $m\geq 2$ and $n\geq 1$, 
    \item $P'_{8\cdot 3^m}=\langle x,y, z \mid x^2=(xy)^2=y^2, zxz^{-1}=y, zyz^{-1}=xy, z^{3^m}=1 \rangle$, where $m\geq 1$.
\end{enumerate}

Since \cref{thm:direct:free:product} allows to calculate the homotopy type of direct products, it's enough to calculate $\Ecom{G}$ for the previous groups.

\begin{theorem}\label{thm:Ecom:Q4n}
    Let $Q_{4n}$ be the finite group of  generalized quaternions of order $4n$. Then $\Ecom{Q_{4n}}\simeq \bigvee_{n^2-1}S^1$. 
\end{theorem}

\begin{proof}
    The group $Q_{4n}$ admits the following presentation, see \cite[p. 98]{BrownBook}:
    \[Q_{4n}=\langle x,y|y^{2n}=1, x^2=y^{n}, x yx^{-1}=y^{-1}\rangle\]
    With this presentation it is easy to check that the elements can be written uniquely as $x^iy^a$ with $i\in\{0,1\}$ and $0\leq a\leq 2n-1$. Following this presentation the only pair of elements that commute are elements that belongs to $\langle y\rangle$ or elements that belongs to $\langle z\rangle$ with $z\notin \langle y\rangle$. Thus the maximal abelian subgroups of $Q_{4n}$ are as follows: the cyclic group $A=\langle y\rangle$ of order $2n$ and for each $z\notin A$ the group $\langle z\rangle$ is a group of order 4. The intersection of any two distinct subgroups is the center $\{1,x^2=y^n\}$, then $\mAbCo{Q_{4n}}$ is of height 1, thus $\mAbCo{Q_{4n}}$ is a graph. There are $2$ cosets of $A$ and for each subgroup of order 4 there are $n$ cosets, since are $n$ of these subgroups in the up level of the poset $\mAbCo{Q_{4n}}$ there are $n^2+2$ vertices. In the bottom level there are $2n$ vertices, the cosets of $Z$. The cosets of $A$ are connected to $n$ cosets of $Z$ and the cosets of a cyclic of order 4 are conected to 2 coset of $Z$, then there are $2n+2n^2$ edges. Using the Euler characteristic, we conclude that the number of circles is
    $$2n^2+2n-(n^2+2+2n)+1=n^2-1.$$
    
\end{proof}

The conclusions of the following theorem where obtained using a GAP routine. Notice that this is possible because we are dealing with two very concrete groups. The code can be found in \Cref{Appendix:code:GAP}.

\begin{theorem}\label{thm:Ecom:P48:P120}
    The following holds
    \begin{itemize}
        \item $\Ecom{P_{48}}\simeq\bigvee_{167}S^1$, and
        \item $\Ecom{P_{120}}\simeq\bigvee_{1079}S^1$.
    \end{itemize}
\end{theorem}

\begin{theorem}\label{thm:Ecom:FiniteDihedral}
    For $m\geq 2$ and $n\geq 1$, $\Ecom{D_{2^m(2n+1)}}\simeq \bigvee_{(2n+1)^2-1}S^1$.
\end{theorem}

\begin{proof}
    First let us prove that $x^2$ is central and that $D_{2^m(2n+1)}/\langle x^2\rangle\cong D_{2(2n+1)}$. From the presentation we have $$x^2yx^{-2}=xy^{-1}x^{-1}=y,$$ so $x^2$ is central. Also
    $$D_{2^m(2n+1)}/\langle x^2\rangle\cong\langle\bar{x},\bar{y}|\bar{y}^{2n+1}=1,\bar{x}\bar{y}\bar{x}^{-1}=\bar{y}^{-1},\bar{x}^2=1\rangle=D_{2(2n+1)}.$$
    Then $D_{2^m(2n+1)}$ fits in the central extension
$$1\to \dbZ/2^{m-1} \to D_{2^m(2n+1)} \to D_{2(2n+1)} \to  1.$$

Recall that $D_{2(2n+1)}$ can be realized as the isometry group of a $(2n+1)$-gon. Since $2n+1$ is odd, every abelian subgroup of $D_{2(2n+1)}$ is generated by a rotation or a reflection, hence the pre-image in $D_{2^m(2n+1)}$ of any abelian subgroup of $D_{2(2n+1)}$ is abelian, then $\mAbCo{D_{2^m(2n+1)}}\cong \mAbCo{D_{2(2n+1)}}.$

The maximal abelian subgroups of $D_{2(2n+1)}$ are as follows: the cyclic group $A=\langle \bar{y}\rangle$ of order $2n+1$ and the reflections. Any two distinct maximal abelian subgroups intersect in the trivial subgroup, hence the height of  $\mAbCo{D_{2(2n+1)}}$ is 1 and then is a graph. The vertices consist of 2 cosets of $A$, $(2n+1)^2$ cosets of all possible reflections and $2(2n+1)$ cosets of the center. The edges are $(2n+1)$ for each coset of $A$ and $2$ for each coset of a reflection, hence the total number of edges is $2(2n+1)+2(2n+1)^2$. Using the Euler characteristic we conclude that the number of circles is
$$2(2n+1)^2+2(2n+1)-(2+(2n+1)^2+(2n+1))+1=(2n+1)^2-1.$$

\end{proof}

\begin{theorem}\label{thm:Ecom:p'}
    For all $m\geq 1$ we have $\Ecom{P'_{8\cdot 3^m}}\simeq\bigvee_{39}S^1$.
\end{theorem}
\begin{proof}
    It is an easy exercise to verify, directly from the presentation of $P'_{8\cdot 3^m}$, that both $z^3$ and $x^2=y^2$ are central elements, thus the subgroup generated by them is a central subgroup (actually it is the center itself). The quotient group has the following presentation
    \[\langle x,y,z | (xy)^2, zxz^{-1}=y, zyz^{-1}=xy\rangle\]
    which no longer depends on $m$. It can be verified, using GAP for example, that this quotient group is $A_4$ the alternating group on four letters. By \Cref{thm:central:extension} it is enough to compute the poset of cosets of abelian subgroups of $A_4$ that have abelian preimage in $P'_{8\cdot 3^m}$. All abelian subgroups of $A_4$ are cyclic of order 3, with the only exception of the Klein subgroup which is noncyclic of order 4. Recall that, since the kernel subgroup is central, the preimage of any cyclic subgroup of $A_4$ is abelian. We claim that the preimage of the Klein subgroup is not abelian. In fact the Klein subgroup is generated by the images of $x$ and $y$, and it can be verified with GAP that they do not commute inside $P'_{24} = P'_{8\cdot 3^1}$ (see \Cref{Appendix:code:GAP}), which is a quotient of $P'_{8 \cdot 3^m}$ for all $m \ge 1$. In conclusion, for all $m\geq 1$, $\Ecom{P'_{8\cdot 3^m}}$ has the homotopy type of the geometric realization of the poset of cosets of \emph{cyclic} subgroups of $A_4$. In particular, all $\Ecom{P'_{8\cdot 3^m}}$ have the same homotopy type and we compute $\Ecom{P'_{24}} \simeq \bigvee_{39} S^1$ using GAP in \Cref{Appendix:code:GAP}.
\end{proof}

%\begin{description}
%\item[Cyclic case] $\dbZ/n$ for each $n\geq 1$.

%\item[Dihedral case] These groups fit in a central extension of the following form
%\[1\to \dbZ/2m \to G \to D_{2n} \to  1\]
%where $m$, $n$ are coprime with $m\geq 1$, $n\geq 2$, and $D_{2n}$ is the dihedral group of order $2n$.

%\item[Tetrahedral case]  These groups fit in a central extension of the following form
%\[1\to \dbZ/2m \to G \to A_4 \to  1\]
%where $m$ is an odd integer with $m\geq 1$, and $A_4$  the alternating group on 4 letters.

%\item[Octahedral case] $\dbZ^m \times O$ where $m$ is coprime with 6 and $O=\langle x,y | (xy)^2=x^3=y^4 \rangle$ is the binary octahedral group which has order 48.

%\item[Icosahedral case] $\dbZ^m \times I$ where $m$ is coprime with 30 and $I=\langle x,y | (xy)^2=x^3=y^5 \rangle$ is %the binary icosahedral group which has order 120.
%\end{description}

\section{3-manifolds modeled on $\hypp$}

Manifolds modeled on $\hypp$, also known as hyperbolic manifolds, may have empty-boundary or not.

\begin{theorem}\label{thm:Ecom:hyperbolic}
Let $G$ be the fundamental group of a finite volume hyperbolic manifold $M$. Then $\Ecom{G}\simeq \bigvee_{\mathbb{N}} S^1$.
\end{theorem}
\begin{proof}
We have two cases depending on whether $M$ is closed ornot. In either case $G$ is a torsion-free group because $M$ is a finite dimensional $K(G,1)$.

If $M$ is closed, then $G$ is a torsion-free word-hyperbolic group that is not virtually cyclic, see \cite[Table 1 on page 19]{AFW15}. Now the assertion follows from \Cref{thm:torsionfree:hyperbolic}. 

Now, let us focus on the case where $M$ is not closed. The following argument is done in the proof of Theorem~3.1 from \cite{LASS21}, nevertheless we include it for the sake of completeness.

 Consider the collection $\mathcal{B}$ of subgroups of $G$ consisting of
 \begin{itemize}
     \item All conjugates of the fundamental groups of the cusps of $M$. All these groups are isomomorphic to $\Z^2$ since every cusp is homemorphic to the product of a torus and an interval.
     
     \item All maximal infinite virtually cyclic of subgroups $G$ that are not subconjugate to the fundamental group of a cusp. All these groups are isomorphic to $\Z$ by the classification of virtually cyclic groups.
 \end{itemize}

 We notice that in \cite[Theorem~2.6]{lafont:ortiz} Lafont and Ortiz verified that every infinite virtually cyclic subgroup of $G$ is contained in an element of $\mathcal B$. On the other hand, if $H$ is a $\Z^2$-subgroup of $G$, then $H$ does not contain a subgroup isomorphic to a free group in two generators, therefore by the Tits alternative for relatively hyperbolic group, see for instance \cite[Remark~3.5]{BW13},  $H$ is subconjugate to the fundamental group of a cusp, that is, $H$ is contained in an element of $\mathcal B$. Also in \cite[Theorem~2.6]{lafont:ortiz} it is proved that the intersection of any two elements of $\mathcal B$ is trivial. Now the result follows from \Cref{thm:countable:wedge}.
\end{proof}

\section{3-manifolds modeled on $\mbs^2\times\mbe$}\label{sec:S2xE}

There are only two possible fundamental groups for manifolds modeled on $\mbs^2\times\mbe$, $\dbZ$ and the infinite dihedral group $D_\infty\cong \dbZ/2 * \dbZ/2$, see for instance Table 1 on page 19 in \cite{AFW15}. In these cases $\Ecom{\dbZ}$ is contractible and, by \Cref{thm:direct:free:product}, $\Ecom{D_\infty}\simeq \bigvee_{\mathbb{N}} S^1$.

\section{3-manifolds modeled on $\hxr$ or $\pslt$}

By \cite[Theorem 4.7.10]{Th97}, we have that any group that appears as the fundamental group of a manifold modeled on 
$\pslt$, also is the fundamental group of a manifold modeled on $\hxr$. Thus we only have to deal with fundamental groups of manifolds modeled on $\hxr$.

\begin{theorem}\label{thm:Ecom:H2xE}
Let $G$ be the fundamental group of a finite volume 3-manifold modeled on $\hxr$. Then $\Ecom{G}\simeq \bigvee_{\mathbb{N}} S^1$.
\end{theorem}
\begin{proof}
Let $M$ be a finite volume 3-manifold modeled on $\hxr$ with fundamental group $G$. Then $M$ is a Seifert fibered space with  base orbifold $B$ modeled on $\hyp$. By \cite[Theorem~2.5.2]{AFW15} we have the following central extension
\[1\to \dbZ \to G \to Q \to1\]
where $\dbZ$ is generated by a regular fiber of $M$, and $Q$ is the orbifold fundamental group of $B$, in particular, $Q$ is a finitely generated fuchsian group.

Recall that for a fuchsian group, all abelian subgroups are (finite or infinite) cyclic. Thus the preimage under the map $G\to Q$ of any  abelian subgroup of $Q$ is an abelian subgroup of $G$. Therefore, by \Cref{thm:central:extension}, $\mAbCo{G}$ is isomorphic, as a poset, to $\mAbCo{Q}$. It follows from standard arguments of plane hyperbolic geometry  that the intersection of any two distinct maximal cyclic subgroups of in a fuchsian group is trivial. The result now follows from  \Cref{thm:countable:wedge}.
\end{proof}

\section{Some crystallographic groups of dimension 1 and 2}

The main objective of this section is to set the stage to deal with the geometries $\mbe^3$, $\nil$ and $\sol$. The fundamental groups arising from these geometries have the common feature of being virtually poly-$\Z$ groups. That is why we set notation and some basic facts about this kind of groups. We also compute the homotopy type of $\Ecom{G}$ for the infinite dihedral group $D_\infty$ and some wallpaper groups, as these will be appear as quotients of the fundamental groups of manifolds modeled on   $\mbe^3$, $\nil$ and $\sol$.

\subsection{Poly-$\Z$ groups}
Recall that a group $G$ is called \emph{a poly-$\dbZ$ of rank $n$} if it admits a filtration $1=G_0<G_1<\cdots < G_n=G$ such that $G_i\mathrel{\unlhd} G_{i+1}$ and $G_{i+1}/G_{i}\cong \dbZ$ for all $i$. It is well known that the rank does not depend on the filtration, that is, is a well-defined invariant of the group that we denote $\mathrm{rank}(G)$. We will use the following well known facts about poly-$\dbZ$ groups that can be found, for instance, in \cite[p. 16]{Seg83}.

\begin{proposition}\label{prop:polyZ}
Let $G$ be a poly-$\dbZ$ group and let $H$ be subgroup. Then the following statements hold
\begin{enumerate}
    \item $H$ is a poly-$\dbZ$ group.
    \item $\mathrm{rank}(H)\leq \mathrm{rank}(G)$.
    \item $\mathrm{rank}(H)= \mathrm{rank}(G)$ if and only if $H$ has finite index in $G$.
\end{enumerate}
\end{proposition}

The proof of the following lemma is left as an exercise to the reader.

\begin{lemma}\label{lemma:mabco:infinite:dihedral}
Let $D_\infty=\dbZ/2 \ast \dbZ/2$ be the infinite dihedral group. Then
\begin{enumerate}
    \item the center of $D_\infty$ is trivial,
    \item the subgroup generated by $ab$, where $a$ and $b$ are generators of each of the factors of $D_\infty$, is the only maximal abelian subgroup of $D_\infty$ of rank 1,
    \item every nontrivial finite  subgroup of $D_\infty$ is conjugate to one of the factors, and
    \item if $A$ and $B$ are two distinct maximal abelian subgroups of $D_\infty$ then $A\cap B$ is trivial.
\end{enumerate}
\end{lemma}

\begin{proposition}\label{lemma:mabco:klein:bottle}
Let $K=\dbZ \rtimes \dbZ$ be the fundamental group of the Klein bottle. Then
\begin{enumerate}
    \item every abelian subgroup of $K$ is free of rank at most 2,
    \item $Z=\{0\} \rtimes 2\dbZ$ is the center of $K$,
    \item $\dbZ \rtimes 2\dbZ$ is the only maximal abelian subgroup of $K$ of rank 2, 
    \item if $A$ and $B$ are two distinct maximal abelian subgroups of $K$ of rank 1 that are not contained in $\dbZ \rtimes 2\dbZ$, then both $A$ and $B$ contain $Z$ as an index 2 subgroup and $A\cap B=Z$,
    \item the intersection of any maximal abelian subgroup of rank 1 with $\dbZ \rtimes 2\dbZ$ equals $Z$, and
    \item $\Ecom{K}\simeq \bigvee_{\mathbb{N}} S^1$.
\end{enumerate}
\end{proposition}
\begin{proof} 
Since $K$ is a torsion free poly-$\dbZ$ group of rank 2, item (1) is clear. Notice that we have a short exact sequence
\[1\to Z \to K \xrightarrow{p} D_\infty\to 1\]
where $D_\infty$ is the infinite dihedral group which it is isomorphic to both $\dbZ \rtimes \dbZ/2$ and $\dbZ/2\ast \dbZ/2$. Note that the center of $ K$ maps into the center of $D_\infty$ which is trivial, thus the center of $K$ is contained in $Z$. It is not difficult to see that $Z$ is contained in the center of $K$, this proves (2). Items (3) and (4) are a direct consequence of \Cref{thm:central:extension} and \Cref{lemma:mabco:infinite:dihedral}. 

Let us prove (5). Let $A$ be a maximal abelian subgroup of $K$. By (4), the center $Z$ is contained in both $A$ and $\dbZ\rtimes 2\dbZ$. Hence $Z\subseteq A\cap (\dbZ\rtimes 2\dbZ)$. If this inclusion is strict, since $Z$ has index 2 in $A$, it would follow that $A$ is contained in $Z\subseteq \dbZ\rtimes 2\dbZ$ which is impossible by maximality of $A$.

Finally, part (6) is a direct consequence of part (5) and \Cref{cor:height:1}. 
%Let us prove (2). Let $A$ be an abelian subgroup of $K$ of rank 2, that is, $A\cong \dbZ^2$. Since $Z$ is abelian of rank 1, $p(A)$ must be an abelian subgroup of rank 1 of $D_\infty$, but $\dbZ\rtimes \{0\}$ is the only maximal abelian subgroup of $D_\infty$. We conclude that $A\subseteq p^{-1}(\dbZ\rtimes \{0\})=\dbZ\rtimes 2\dbZ$.

%Clearly, any maximal abelian subgroup contains the center. If $A$ is a maximal abelian subgroup of $K$ of rank 1, then $p(A)$ is a finite abelian subgroup of $D_\infty$, but all finite subgroups of $D_\infty$ are subconjugate to one of the factors in the free product decomposition, this implies that $A$ contains $Z$ as an index 2 subgroup. If $B$ is another maximal abelian subgroup of $K$, then $A\cap B$ is a proper subgroup of both $A$ and $B$ that contains $Z$, therefore the intersection equals $Z$ as it has index 2 in $A$ and $B$. This finishes the proof of (3).
\end{proof}

Consider the following matrices:

\begin{equation}\label{eq:matrices:2crystallographic}
    A_2=\begin{pmatrix}
    -1&0\\
    0&-1
    \end{pmatrix},     A_3=\begin{pmatrix}
    0&-1\\
    1&-1
    \end{pmatrix},
    A_4=\begin{pmatrix}
    0&-1\\
    1&0
    \end{pmatrix},A_6=\begin{pmatrix}
    1&-1\\
    1&0
    \end{pmatrix}
\end{equation}

\begin{proposition}\label{prop:2dim:crystallographic}
    Let $Q=\dbZ^2\rtimes_{A_n}\dbZ/n$ where $A_n$ is one of the matrices in \cref{eq:matrices:2crystallographic}. Then 
\begin{enumerate}
    \item every element in $Q-(\dbZ^2\rtimes \{0\})$ has finite order,
    \item every finite subgroup of $Q$ is cyclic,
    \item the intersection of any two finite maximal abelian subgroups of $Q$ is trivial,
    \item $\dbZ^2\rtimes \{0\}$ is the only infinite maximal abelian subgroup of $Q$,
    \item $\Ecom{Q}\simeq \bigvee_{\mathbb{N}} S^1$.
\end{enumerate}
\end{proposition}
\begin{proof}
  This group $Q$ is one of the plane crystallographic groups, generated by translations by vectors in $\dbZ^2$ and a single rotation around the origin of order $n$. Every element in the group is either a translation or a rotation of order dividing $n$ around some center. It is easy to see when two elements commute: two translation always commute, two rotations only commute if they have the same center, and a translation never commutes with a rotation. This implies that the maximal abelian subgroups are the group of translations, and for each possible center of rotation, the group of all rotations with that center. Parts (1)--(4) follow from this description of the maximal abelian subgroups.

  Part (5) is a direct consequence of parts (3) and (4) and \Cref{cor:height:1}.
\end{proof}

\section{A little technical detour}

In this section we settle notation and a preliminary result concerning the structure of gropus of the form $K_1\ast_A K_2$ where $K_1$ and $K_2$ are copies of the fundamental group of the Klein bottle, $A$ is the index two $\dbZ^2$ subgroup of $K_1$ and $K_2$. These groups will appear in the analysis of geometries $\mbe^3$, $\nil$, and $\sol$, and it is fair to say that they are the most difficult cases in this article.

\begin{notation}\label{notation:sol}
Let us set some notation. Let $G$ be a group of the form  $K_1\ast_A K_2$ where $K_1$ and $K_2$ are copies of the fundamental group of the Klein bottle, $A$ is the index two $\dbZ^2$ subgroup of $K_1$ and $K_2$, and the amalgamation is given by an automorphism $\varphi\colon A\to A$. Since $A$ is normal in both factors of $G$, we conclude $A$ is normal in $G$. Moreover,  $G/A$ is isomorphic to the infinite dihedral group $D_\infty$. Let $p$ denote the quotient map. Now, $D_\infty$ is isomorphic to $\dbZ\rtimes_{-1} \dbZ/2$, and this implies that we have an index 2 subgroup $G'=p^{-1}(\dbZ\rtimes \{0\})$ of $G$ that is isomorphic to $A\rtimes_{\varphi'} \dbZ$ where $\varphi'$ is a certain automorphism of $A$ that depends on $\varphi$. All of this can be summarized in the following commutative diagram
\[
\xymatrix{
1 \ar[r]  & A \ar@{=}[d] \ar[r] & G \ar[r]^{p} & \dbZ\rtimes_{-1} \dbZ/2 \ar[r] & 1 \\
1 \ar[r] & A \ar[r] & G' \ar[r]^{p} \ar@{^{(}->}[u] & \dbZ \ar[r] \ar@{^{(}->}[u] & 1
}
\]

Depending on the automorphism $\varphi$, $A$ may or may not be a maximal abelian subgroup of $G$. We will classify the maximal abelian subgroups of $G$ that are not equal to $A$ (which, if $A$ is not maximal, are all of them) as follows:
    \begin{description}
        \item[Type I] those that are subconjugate to one of the factors, or equivalently, those that have nontrivial finite image under $p$, and
        \item[Type II] those that are subgroups of $G'$, equivalently those that have infinite image under $p$.
    \end{description}
\end{notation}

Given an element $x\in G$ we denote $p(x)$ by $\bar x$.

\begin{lemma}\label{lemma:conjugate:factors}
    Consider  $K_1\ast_A K_2$ where $K_1$ and $K_2$ are copies of the fundamental group of the Klein bottle, $A$ is the index two $\dbZ^2$ subgroup of $K_1$ and $K_2$. Let $x$ and $y$ be elements in $K_1-A$ and in $K_1-A$ respectively. Then
    \begin{enumerate}
        \item every conjugate of $K_i$ for $i=1,2$ is of the form $K_i^{(xy)^n}$ with $n\in \dbZ$,
        \item the center of $K_1^{(xy)^n}$ (resp. $K_2^{(xy)^n}$) is the infinite cyclic group generated by $(x^2)^{(xy)^n}$ (resp. $(y^2)^{(xy)^n}$),
        \item any abelian subgroup of $K_1\ast_A K_2$ that is maximal among those of type II, is maximal abelian in $K_1\ast_A K_2$.
    \end{enumerate}
\end{lemma}
\begin{proof}
    Let us prove (1), and note that (2) is a direct consequence of (1).

    Consider $p\colon K_1\ast_A K_2\to D_\infty = \langle \bar x \rangle \ast \langle \bar y \rangle$. Note that $K_1^g=p^{-1}(\langle \bar x \rangle^{\bar g})$, the result now follows from the fact that all conjugates of $\langle \bar x \rangle$ in $D_\infty$ are of the form $\langle \bar x \rangle^{(\bar x\bar y)^n}$ with $n\in \dbZ$.

    To prove (3), let $B$ be an abelian subgroup of $K_1\ast_A K_2$ that is maximal among those of type II, then $p(B)$ is infinite. Let $C$ be an abelian subgroup of $K_1\ast_A K_2$ that contains $B$, then $p(C)$ is infinite, and therefore $C$ is of type II. We conclude $B=C$ by maximality.
\end{proof}

\section{3-manifolds modeled on $\mbe^3$}

There are exactly 6 orientable compact 3-manifolds modeled on $\mbe^3$ and their fundamental groups are the following, see for instance \cite[Theorem 3.5.5]{Wolf}.

\begin{itemize}
    \item $G_n=\dbZ^2\rtimes_{\varphi_n} \dbZ$, where $\varphi_n$ is one of the following matrices:     \[\varphi_1=\begin{pmatrix}
    1&0\\
    0&1
    \end{pmatrix},     \varphi_2=\begin{pmatrix}
    -1&0\\
    0&-1
    \end{pmatrix},
    \varphi_3=\begin{pmatrix}
    0&1\\
    -1&0
    \end{pmatrix},\varphi_4=\begin{pmatrix}
    0&-1\\
    1&-1
    \end{pmatrix},\varphi_6=\begin{pmatrix}
    0&-1\\
    1&1
    \end{pmatrix}
\]
of orders 1, 2, 3, 4, and 6 respectively.

\item The sixth group, which we denote $\Gamma_6$, admits the  presentation:
\[\langle x,y,z | x y^2 x^{-1}=y^{-2} , y x^2 y^{-1}=x^{-2}, z=y^{-1}x^{-1} \rangle.\]

Notice that the subgroup $T$ generated by $x^2$, $y^2$ and $z^2$ is the translation subgroup of $\Gamma_6$, and therefore is isomorphic to $\Z^3$. Also notice that the subgroups $K_x = \langle x, y^2 \rangle$ and $K_y = \langle x^2, y \rangle$ are isomorphic to the fundamental group of the Klein bottle and we have the following splitting
\[\Gamma_6\cong  K_x\ast_{A} K_y\]
where $A$ is the subgroup generated by $x^2$ and $y^2$. It is not difficult to see that, following  \Cref{notation:sol}, $\Gamma_6'$ is isomorphic to $A\rtimes_{-1} \langle \bar x \bar y \rangle$=$A\rtimes_{-1} \langle \bar z \rangle$.
\end{itemize}

\begin{theorem}\label{thm:G_n}\label{thm:Ecom:G1-G5}
\begin{enumerate}
    \item $\Ecom{G_1}$ is contractible,
    \item every abelian subgroup of $G_n$ of rank at least 2, with $n\neq 1$, is contained in $\dbZ^2 \rtimes_{\varphi_n} n\dbZ\cong \dbZ^3$,
    \item  the center $Z$ of $G_n$, for $n\neq 1$, equals $\{0\}\rtimes_{\varphi_n} n\dbZ\cong \dbZ$, and the intersection of any two distinct maximal abelian subgroups of $G_n$ equals $Z$,
    \item $\Ecom{G_n}\simeq \bigvee_{\mathbb{N}} S^1$ for $n=2,3,4,6$. 
\end{enumerate}
\end{theorem}
\begin{proof}
The conclusion in item (1) is clear since $G_1$ is abelian. For the rest of the proof, let $n=2,3,4,6$  and, for brevity, set $\varphi=\varphi_n$, and let $G=G_n$.

Let us prove part (2). Let $B$ be an abelian subgroup of $G$ of rank at least 2 such that $Z\leq B$ (see \Cref{remark:maximal:contains:center}). We have the following short exact sequence 
\[1\to B\cap (\dbZ^2\rtimes_\varphi \{0\}) \to B \to  m\dbZ \to 1\] where $m$ divides $n$.
Thus $B$ is generated by $B\cap (\dbZ^2\rtimes_\varphi \{0\})$ and a preimage $x$ of a generator of the quotient $m\dbZ$. Moreover, since $x$ is of the form $(z,m)$ and $B$ is abelian, we have that $(0,m)$ acts trivially by conjugation on $B\cap (\dbZ^2\rtimes_\varphi \{0\})$. Since $B$ has rank 2, we have that $B\cap (\dbZ^2\rtimes_\varphi \{0\})$ is nontrivial, and thus $n=m$ because $\varphi^m$ does not have fixed points unless $n$ divides $m$. Hence in this case, we have that $B$ is of the form $(B\cap (\dbZ^2\rtimes_\varphi \{0\}))\times n\dbZ$. If additionally, $B$ is a maximal abelian subgroup of $G$, then $B=\dbZ^2 \times n\dbZ$.

Now we proceed with the proof of part (3). The first claim about the center is left as an exercise to the reader, and we will show the part about the intersection of maximal abelian subgroups. We claim that if $B$ and $C$ are two distinct maximal abelian subgroups satisfying that their intersection with $\dbZ^2\rtimes_\varphi \{0\}$ is trivial, then $B\cap C=Z$. Consider the quotient $G/Z=\dbZ^2\rtimes_\varphi (\dbZ/n\dbZ)$. One can easily see that, from the definition of  $\varphi$, the action of $\dbZ/n\dbZ$ on $\dbZ^2-\{(0,0)\}$ is free, hence by \cite[Lemma~6.3(i)]{LS00}, $G/Z$ satisfies that the intersection of any two distinct maximal finite (actually cyclic) subgroups of $G/Z$ is trivial. Now the claim follows by the correspondence theorem applied to the projection $G\to G/Z$.

To finish the proof of this item, we will prove that the intersection of any two distinct maximal abelian subgroups of $G/Z$ is trivial and apply \Cref{cor:height:1}. From the previous paragraph we only have to show that the intersection of $T=\dbZ^2 \times n\dbZ$ with any other maximal abelian subgroup $B$ (those described in the previous paragraph) is exactly $Z$. Since $B \cap \dbZ^2 = 0$, we have that $T \cap B \le Z$, and since they are maximal, $T \cap B \ge Z$ too.
\end{proof}

We now turn our attention to the abelian subgroups of $\Gamma_6$. In the following lemma we describe the maximal abelian subgroups of $\Gamma_6$ and their intersections, all the conclusions in this lemma are summarized in \Cref{fig:G6}.

\begin{lemma}\label{lem:ab:subgps:G6}
Following \Cref{notation:sol} we have that:
\begin{enumerate}
    \item every abelian subgroup of $\Gamma_6$ is free of rank at most 3,
    \item the subgroup $T=\langle x^2, y^2, z^2 \rangle$ is the only rank 3 maximal abelian subgroup of $\Gamma_6$, 
   \item every rank 2 abelian subgroup of $G$ is contained in $T$, in particular there are no maximal abelian subgroups of rank 2, 
    \item the intersection of any two subgroups of Type II is $\langle z^2 \rangle$,
    \item  the center of any conjugate of $K_x$ (resp. $K_y$) is $\langle x^2 \rangle$ (resp. $\langle y^2 \rangle$),
    \item the intersection of any two subgroups of Type I that lie in conjugates of $K_x$ (resp. $K_y$) is $\langle x^2 \rangle$ (resp. $\langle y^2 \rangle$),
    \item the centers of any conjugate of $K_x$ and any conjugate of $K_y$ intersect trivially,
    \item the intersection of any subgroup of Type I with any subgroup of Type II different from $T$ is trivial,

    \item the intersection of $T$ with any subgroup of Type I is either $\langle x^2 \rangle$ or $\langle y^2 \rangle$.
\end{enumerate}
\end{lemma}
\begin{proof}
\begin{enumerate}
        \item Since $G$ is a torsion free poly-$\dbZ$ group of rank 3, any abelian subgroup is free of rank at most 3, see \cref{prop:polyZ}.
        \item For any $n$-dimensional crystallographic group, the translation group is the only maximal abelian subgroup of rank $n$. 
        \item Let $B$ be a rank 2 abelian subgroup of $\Gamma_6$. If $p(B)$ is finite, then $B$ is subconjugate to either $K_x$ or $K_y$, thus $B$ must be contained in $A\subseteq T$. If $p(B)$ is infinite, then $B$ is contained in $\Gamma_6'$. Thus by \cref{thm:G_n} (2) $B$ is contained in $T$.
        \item It is a direct consequence of the fact that $\Gamma_6'$ is isomorphic to $G_2$ and \Cref{thm:G_n} (3).
        \item By \cref{lemma:conjugate:factors} (1) all conjugates of $K_x$ (resp. $K_y$) are of the form $K_x^{z^n}$ (resp. $K_y^{z^n}$) for some $n\in \dbZ$. It is easy to check that $(x^2)^z = x^{-2}$ and $(y^2)^z = y^{-2}$ directly from the presentation. The claim now follows from \cref{lemma:conjugate:factors}(2).
        \item This is a consequence of the previous item and \Cref{lemma:mabco:klein:bottle} (4).
        \item By item (6) we have, for all $n$, that  $Z(K_x)^{z^n}\cap Z(K_y)^{z^n}=\langle x^2\rangle\cap \langle y^2\rangle$, and the latter intersection is clearly trivial since $x^2$ and $y^2$ are independent generators of the translation group $T$.
        
        \item Let $B$ a subgroup of type I  and $C$ a subgroup of Type II different from $T$.   By parts (2) and (3) both $B$ and $C$ must be infinite cyclic. By part (6) $B$ either contains $\langle x^2\rangle$ or $\langle y^2\rangle$ as a finite index subgroup, and by part (4) $C$ contains $\langle z^2\rangle$ as a finite index subgroup. The claim follows by noting that $\langle x^2\rangle\cap \langle z^2\rangle$ and $\langle y^2\rangle\cap \langle z^2\rangle$ are trivial.

        \item Let $B$ be a maximal abelian subgroup of Type I, which means that $p(B)$ is finite. Since $p(T)$ is infinite cyclic, we must have $p(B \cap T) = 0$, so that $B \cap T \le \ker(p) = \langle x^2, y^2 \rangle$. By \cref{lemma:mabco:klein:bottle} (5), the intersection of $B$ with $\langle x^2, y^2 \rangle$ is the center of the conjugate of $K_x$ or $K_y$ containing $B$, and by part (5) this is either $\langle x^2 \rangle$ or $\langle y^2 \rangle$.

\end{enumerate}
\end{proof}
\begin{figure}
    \centering
    \includegraphics{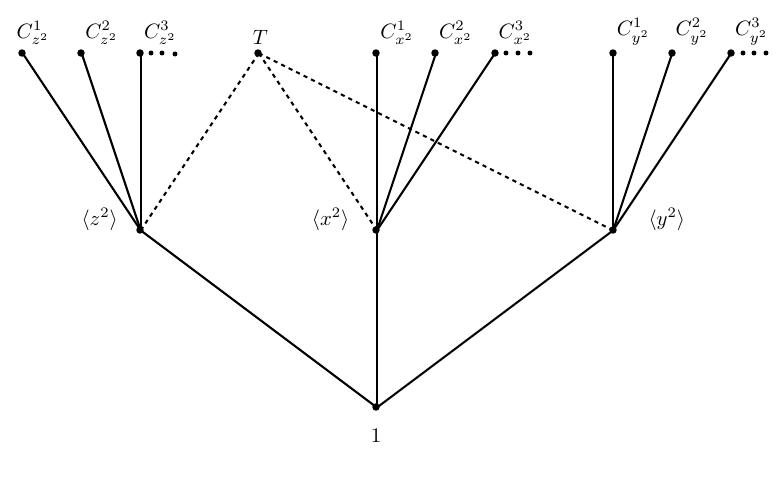}
    \caption{\small The Hasse diagram of $\mAb{\Gamma_6}$. Subgroups of the form $C^i_{z^2}$ are infinite cyclic of Type II. Subgroups of the form $C^i_{x^2}$ and $C^i_{y^2}$ are infinite cyclic of Type I.}
\label{fig:G6}
\end{figure}

The following corollary is an immediate consequence of the preceding lemma.

\begin{corollary}\label{Cor:chains:G6} Following \Cref{notation:sol} and \Cref{lem:ab:subgps:G6}, every maximal chain in $\mAb{G}$ is of one of the following types:
    \begin{enumerate}
        \item $1 < \langle w \rangle< T$ where $w$ is either $x^2$, $y^2$ or $z^2$,
        \item $1 < \langle w \rangle< C_w$  where $w$ is either $x^2$, $y^2$ or $z^2$ and $C_w$ is an infinite cyclic group that contains $\langle w \rangle$ as a subgroup of index 2,
    \end{enumerate}
\end{corollary}

\begin{theorem}\label{thm:Ecom:G6}
We have $\Ecom{\Gamma_6}\simeq \bigvee_{\mathbb{N}} S^1$.
\end{theorem}

\begin{proof}
As a consequence of \Cref{Cor:chains:G6}, all maximal chains in $\mAbCo{G}$ are of one of the following types:
    \begin{enumerate}
        \item $\{g\}< g\langle w \rangle< gT$ where $w$ is either $x^2$, $y^2$ or $z^2$,
        \item $\{g\} < g\langle w \rangle< gC_w$  where $w$ is either $x^2$, $y^2$ or $z^2$ and $C_w$ is an infinite cyclic group that contains $\langle w \rangle$ as a subgroup of index 2,
    \end{enumerate}
Note that for a given choice of $g \in G$ and subgroup $C_w$ of Type I, there is a unique triangle of the form (2) in the above list, so we can remove the edge ${g} < gC_w$ and the interior of the triangle without changing the homotopy type. After removing those triangles, each edge of the form $\{g\} < g \langle w \rangle$ is in a unique triangle, which is of the form (1). We can remove those edges and the interior of their containing triangles without changing the homotopy type either.

At this point, we have no triangles remaining, which shows $\mAbCo{G}$ has the homotopy type of a graph, and hence of a wedge of circles (since it is connected). The complex obtained after removing the triangles still contains a copy of $\mAbCo{G'}$, which has the homotopy type of a countable wedge of circles, and therefore so does $\mAbCo{G}$.
\end{proof}

\section{3-manifolds modeled on $\nil$}

A very explicit classification of fundamental groups of manifolds modeled on $\nil$ can be found in \cite{DIKK93}. In that paper, they split the classification into seven infinite families that they call Type 1-7 . For convenience we collect the types into three clusters.

\begin{description}
\item[Semi direct products] Groups of the form $\dbZ^2\rtimes_\varphi \dbZ$ where $\varphi$ is a matrix of the form $\begin{pmatrix}
    1&n\\
    0&1
    \end{pmatrix}$ with $n\geq 1$. This corresponds to Type 1 groups in \cite{DIKK93}.
\item[Orientable crystallographic quotient] Groups $G$ that fit in a central extension of the form
\begin{equation}\label{critallographic:quotient}
1\to \dbZ \to G  \to \dbZ^2\rtimes_{A_n} \dbZ/n \to 1 \end{equation}
where $A_n$ is one of the matrices in \cref{eq:matrices:2crystallographic}. These groups correspond to Types 2, 5, 6, and 7 in \cite{DIKK93}.
    \item[Non-orientable crystallographic quotient] We subdivide these groups into two Types that we will call Type pg and Type pgg respectively. Groups of Type pg are defined by means of the following presentation
    \[E_k=\langle a,b,c, \alpha | [a,b]=c^{2k}, \alpha c = c^{-1}\alpha, [c,a]=[c,b]=1, \alpha b = b^{-1}\alpha c^{-k}, \alpha^2=a \rangle\]
    while groups of Type pgg are defined by means of the following presentation
    \begin{align*}
    \Gamma_k=\langle a,b,c, \alpha,\beta |& [a,b]=c^{2k}, [c,a]=1,[c,b]=1, \alpha c = c^{-1}\alpha, \beta a = a^{-1}\beta c^{k}, \beta b=b^{ -1}\beta c^{-k},\\
    & \beta^2=c,\alpha^2=a, \alpha b=b^{-1}\alpha c^{-k}, \beta \alpha=a^{-1}b^{-1}\alpha\beta c^{-k-1}\rangle
\end{align*}
where $k$ in each case is a positive integer. These groups correspond to Types 3 and 4 in \cite{DIKK93}.
\end{description}

\subsection{Semi direct products}
\begin{theorem}\label{thm:nil:semidirect}
Let $G=\dbZ^2\rtimes_\varphi \dbZ$ where $\varphi$ is a matrix of the form $\begin{pmatrix}
    1&n\\
    0&1
    \end{pmatrix}$ with $n\geq 1$. Then 
    \begin{enumerate}
        \item every maximal abelian subgroup of $G$ has rank 2,
        \item the intersection of any two maximal subgroups equals $Z(G)$,
        \item     $\Ecom{G}\simeq \bigvee_{\mathbb{N}} S^1$.
    \end{enumerate}
    
\end{theorem}
\begin{proof}
Consider the central extension
\[1\to Z \to G \to Q \to 1\]
where $Z= \dbZ $, and the first map is given by $n\mapsto (n,0,0)$ and the second map is the quotient projection. Note that $Q$ is isomorphic to $\dbZ^2$. 

Notice that  $G$ is a poly-$\dbZ$ group of rank $3$ that does not contain a $\dbZ^3$-subgroup, see \cite[Table 1]{AFW15}. On the other hand, it is not difficult to verify that $Z$ is the center of $G$. Since $Q$ is isomorphic to $\dbZ^2$, and the preimage in $G$ of any infinite cyclic subgroup of $Q$ is abelian of rank 2, we conclude that every maximal abelian subgroup of $G$ is isomorphic to $\dbZ^2$ and it contains $Z$. This proves (1). Let us call $p$ the surjective homomorphism $G\to Q$. As $Q$ is torsion free, then any maximal abelian subgroup of $G$ isomorphic to $\dbZ$ must be equal to $Z$. Let $B,C\cong \dbZ^2$ be two maximal abelian subgroup of $G$, then $p(B)$ and $p(A)$ are maximal abelian subgroups of $Q$ that are isomorphic to $\dbZ$. Hence either $p(A)=p(B)$ or $p(A)\cap p(B)$ is trivial. This proves (2). Now (3) follows directly from \Cref{cor:height:1}.
\end{proof}

\subsection{Orientable crystallographic quotient}
\begin{theorem}\label{thm:nil:orientable:orbifold}
Let $G$ be the fundamental group of a 3-manifold modeled on $\nil$ and such that it fits in a central extension of the form
\begin{equation}
1\to \dbZ \to G  \to \dbZ^2\rtimes_{A_n} \dbZ/n \to 1 \end{equation}
where $A_n$ is one of the matrices in \cref{eq:matrices:2crystallographic}.
 Then $\Ecom{G}\simeq \bigvee_{\mathbb{N}} S^1$.
\end{theorem}
\begin{proof}
Since the extension is central, we can use \Cref{thm:central:extension} to translate  the problem to studying the poset generated by the maximal abelian subgroups of $\dbZ^2\rtimes_{A_n} \dbZ/n$ that have abelian preimage in $G$.

By \Cref{prop:2dim:crystallographic} we have that every abelian subgroup of $\dbZ^2\rtimes_{A_n} \dbZ/n$ is either finite cyclic or a subgroup of $\dbZ^2\rtimes_{A_n} \{0\}$. The preimage of (any finite index subgroup of) $\dbZ^2\rtimes_{A_n} \{0\}$ in $G$ is clearly nonabelian, otherwise $G$ contains a $\Z^3$ subgroup which is impossible in the presence of $\nil$-geometry. Hence we get that the only abelian subgroups of $\dbZ^2\rtimes_{A_n} \dbZ/n$ with abelian preimage are the (infinite or finite) cyclic  ones. Clearly the intersection of any two of these maximal distinct cyclic subgroups is trivial. Therefore the intersection of any two distict maximal abelian subgroups of $G$ is the kernel group $\dbZ$; now the claim follows from \Cref{cor:height:1}.
\end{proof}

\subsection{Case pg}

This is a family of groups, one for each positive integer $k$, and they correspond to Type 3 in \cite{DIKK93}. We have the following presentation

\[E_k=\langle a,b,c, \alpha | [a,b]=c^{2k}, \alpha c = c^{-1}\alpha, [c,a]=[c,b]=1, \alpha b = b^{-1}\alpha c^{-k}, \alpha^2=a \rangle\]

\begin{lemma}
    Let $k$ be a positive integer. Then 
    $E_k$ fits in the following short exact sequence
    \[1\to \langle a,c \rangle \to E_k \to \langle \bar \alpha \rangle\ast \langle \bar{b\alpha} \rangle \to 1\]
    In particular we have a splitting $E_k= K_\alpha \ast_A K_{b\alpha}$, where $K_\alpha = \langle c, \alpha \rangle$, $K_{b\alpha} = \langle c, b\alpha \rangle$, and $A=\langle a,c \rangle$.

    Moreover, following \Cref{notation:sol}, $E_k'$ is isomorphic to $A\rtimes_\varphi \dbZ$ where $\varphi$ is the automorphism of $A$ given by the matrix $\begin{pmatrix}
    1&2k\\
    0&1
    \end{pmatrix}$.
\end{lemma}

\begin{proof}
    The relations that define $E_k$ imply that the subgroup $\langle c,a\rangle$ is stable under conjugation for $b$ and $\alpha$, therefore $\langle a,c\rangle$ is a normal subgroup of $E_k$. It is clear that the quotient group $p(E_k)=E_k/\langle a,c\rangle$ is generated by $\bar{\alpha}$ and $\bar{b\alpha}$. Moreover $p(E_k)$ admits the following presentation $\langle \bar{\alpha},\bar{b\alpha}| (\bar{\alpha})^2=1, \bar{b\alpha}^2=1\rangle$ because $b\alpha b\alpha=bb^{-1}\alpha c^{-k}\alpha=\alpha c^{-k}\alpha=c^k\alpha^2=c^k a$, thus $p(E_k)= \langle \bar{\alpha}\rangle\ast\langle \bar{b\alpha}\rangle\cong D_\infty$. Since $\alpha^2=a$ and $\alpha c \alpha^{-1}=c^{-1}$, $A$ is a index 2 subgroup of $\langle c, \alpha\rangle$ and this group torsion free and non abelian, thus $\langle c,\alpha\rangle=K_\alpha$ is the fundamental group of a Klein bottle, in a similar way $\langle c, b\alpha\rangle=K_{b\alpha}$ is the fundamental group of a Klein bottle.

    To finish, from the presentation, we have that $(b\alpha^2)c(b\alpha^2)^{-1}=c$ and  $(b\alpha^2)a(b\alpha^2)^{-1}=b \alpha^2 b^{-1}=b a b^{-1}=c^{2k}a$ so the automorphism $A\to A$ given by conjugation by $b\alpha^2$ is represented by the matrix  $\begin{pmatrix}
    1&2k\\
    0&1
    \end{pmatrix}$ with respect to the ordered base $c,a$.
\end{proof}

In the following lemma we describe the maximal abelian subgroups of $E_k$ and their intersections, all the conclusions in this lemma are summarized in \Cref{fig:pg}.

\begin{lemma}\label{lem:ab:subgps:pg}
Following \Cref{notation:sol} we have that:
\begin{enumerate}
    \item every abelian subgroup of $E_k$ is free of rank at most 2,
    \item every maximal abelian subgroup of Type II has rank 2 and the intersection of any two of such subgroups is $\langle c \rangle$,
    \item  the center of $K_\alpha^{(ba)^n}$ (resp. $K_{b\alpha}^{(ba)^n}$) is $\langle c^{2nk}a \rangle$ (resp. $\langle c^{(2n+1)k}a \rangle$),

    \item let $i,j \in \{ \alpha, b\alpha \}$, if $K_i^{(ba)^n}\neq K_j^{(ba)^m}$, then $Z(K_i^{(ba)^n})\cap Z(K_j^{(ba)^m})$ is trivial,

    \item the intersection of any two subgroups of Type I that lie in a single conjugate $K_i^{(ba)^n}$ is $Z(K_i^{(ba)^n})$,

    \item the intersection of two subgroups of Type I is trivial if  they belong to distinct conjugates of the factors of $E_k$,

    \item the intersection of $A$ with a subgroup of Type I that is subconjugate to $K_i^{(ba)^n}$ is $Z(K_i^{(ba)^n})$,

    \item the intersection of $A$ with a subgroup of Type II is $\langle c \rangle$.

    \item the intersection of any subgroup of Type I with any subgroup of Type II is trivial,

\end{enumerate}
\end{lemma}
\begin{proof}
\begin{enumerate}
    \item Since $E_k$ is a torsion free virtually poly-$\dbZ$ group of rank 3, any abelian subgroup is free of rank at most 3. On the other hand $E_k$ cannot have a $\dbZ^3$ subgroup as this posibility only happens, in the context of 3-manifolds, in the presence of an euclidean metric.
    \item Every subgroup of Type II is contained in $E_k'$. The claim follows directly from \Cref{thm:nil:semidirect}.

\item By \Cref{lemma:conjugate:factors} (1) all conjugates of $K_\alpha$ (resp. $K_{b\alpha}$) are of the form $K_\alpha^{(ba)^n}$ (resp. $K_{b\alpha}^{(ba)^n}$) for some $n\in \dbZ$. It is easy to check that $(\alpha^2)^{ba} = c^{2k}a$ and $((b\alpha)^2)^{ba} = (c^{k}a)^{ba}=c^{3k}a$ directly from the presentation. The claim now follows from \cref{lemma:conjugate:factors}(2).

\item By the previous item, the centers $Z(K_i^{(ba)^n})$ and $Z(K_j^{(ba)^m})$ are of the form $\langle c^{j_nk}a \rangle$ and $\langle c^{j_mk}a \rangle$ for some $j_n$ and $j_m$. It is easy to check that $j_n\neq j_m$, whether or not $i=j$, as long as $Z(K_i^{(ba)^n}) \neq Z(K_j^{(ba)^m})$. Thus $\langle c^{j_nk}a \rangle \cap\langle c^{j_mk}a \rangle=1$ (we can think these groups as lines in the lattice plane $\dbZ^2 \cong \langle c, a \rangle$ with different slopes). 
\item This is (3) in \Cref{lemma:mabco:klein:bottle}.
\item Let $B$, $C$ be Type I subgroups that lie in different conjugates $K_i^{(ba)^n}$ and $K_j^{(ba)^m}$, respectively. Then by \Cref{lemma:mabco:klein:bottle}  (4), $Z(K_i^{(ba)^n})$ and $Z(K_j^{(ba)^m})$ are index 2 subgroups of $B$ and $C$, respectively. By part (4), the centers have trivial intersection, and thus so do $A$ and $B$.

\item This is (4) in \cref{lemma:mabco:klein:bottle}.

\item Since $A$ is maximal in $E_k'$ and $Z(E_k')=\langle c \rangle$, this item is a direct consequence of \Cref{thm:nil:semidirect} (2).

\item Let $B$ be a Type I subgroup and $C$ a Type II subgroup, this means that $p(B)\cong\dbZ/2$ and $p(C)\cong \dbZ$. Then $p(B\cap C)\subset p(B)\cap p(C)=1$, thus $B\cap C\subset A=\langle c,a\rangle$. Therefore $B\cap C=B\cap C\cap A=(B\cap A)\cap(C\cap A)= Z(K_i^{(ba)^n})\cap \langle c\rangle= \langle c^{j_nk} a\rangle\cap \langle c\rangle=1 $.
\end{enumerate}
\end{proof}
\begin{figure}
    \centering
    \includegraphics{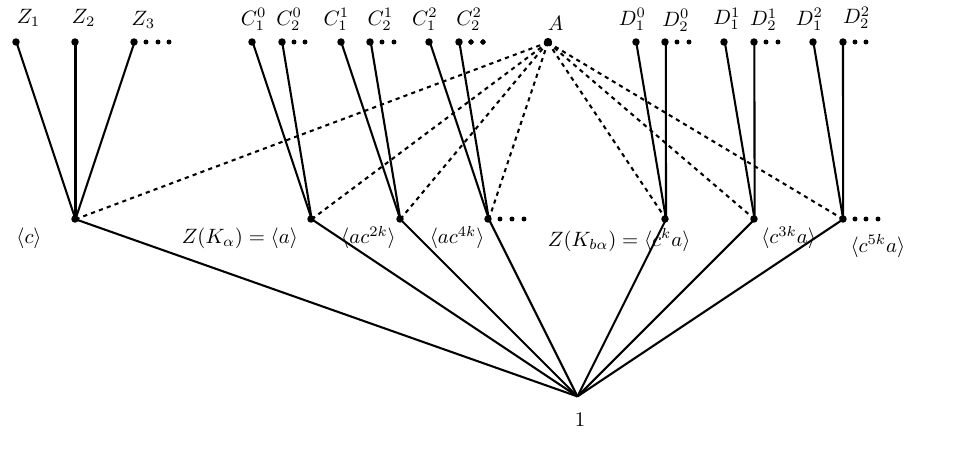}
\caption{\small The Hasse diagram of  $\mAb{E_k}$. Subgroups of the form $Z_i$ are exactly those of Type II. Subgroups in the middle level other than $\langle c \rangle$  are exactly the centers of conjugates of $K_\alpha$ and $K_{b\alpha}$. Subgroups of the form $C^i_j$ and $D^i_j$ are of Type I.}
    \label{fig:pg}
\end{figure}
    
The following corollary is an immediate consequence of the preceding lemma.

\begin{corollary}\label{Cor:chains:pg} Following \Cref{notation:sol} and \Cref{lem:ab:subgps:pg}, every maximal chain in $\mAb{G}$ is of one of the following types:
    \begin{enumerate}
        \item $1 < \langle c \rangle< A$ where $A=\langle a,c \rangle$,
        \item $1 < \langle c^{mk}a \rangle< A$  where $m\in \dbZ$,
        \item $1 < \langle c \rangle< B$ where $B$ is a subgroup of Type II,
        \item $1 < \langle c^{mk}a \rangle< C$  where $m\in \dbZ$  and $C$ is an infinite cyclic group that contains $\langle c^{mk} \rangle$ as a subgroup of index 2,
    \end{enumerate}
\end{corollary}

\begin{theorem}\label{thm:Ecom:nil:pg} Let $k$ be a positive integer, then $\Ecom{E_k}\simeq \bigvee_{\mathbb{N}} S^1$.
\end{theorem}

\begin{proof}As a consequence of \Cref{Cor:chains:pg}, all maximal chains in $\mAbCo{E_k}$ are of one of the following types:
    \begin{enumerate}
        \item $\{g\}< g\langle w \rangle< gA$ where $w$ is either $c$ or $c^{mk}a$,
        \item $\{g\}< g\langle c \rangle< gB$ where $B$ is a subgroup of Type II,
        \item $\{g\} < g\langle ac^{mk} \rangle< gC$ where $m\in \dbZ$ and $C$ is a subgroup of Type I.
    \end{enumerate}
Note that for a given choice of $g \in E_k$, subgroup $C$ of Type I and subgroup $B$ of Type II there are unique triangles of the forms (3) and (2), respectively, so we can remove the edges ${g} < gC$ and $g<gB$ and the interiors of the triangles without changing the homotopy type. After removing those triangles, each edge of the form $\{g\} < g \langle ac^{mk} \rangle$ or $\{g\}<g\langle c\rangle$ is in a unique triangle, which is of the form (1). We can remove those edges and the interior of their containing triangles without changing the homotopy type either.

At this point, we have no triangles remaining, which shows $\mAbCo{E_k}$ has the homotopy type of a graph, and hence of a wedge of circles (since it is connected). The complex obtained after removing the triangles still contains a copy of $\mAbCo{E_k'}$, which has the homotopy type of a countable wedge of circles, and therefore so does $\mAbCo{E_k}$.
\end{proof}

\subsection{Case pgg}
This is a family of groups, one for each positive even integer $k$, and they correspond to Type 4 in \cite{DIKK93}. We have the following presentation

\begin{align*}
    \Gamma_k=\langle a,b,c, \alpha,\beta |& [a,b]=c^{2k}, [c,a]=1,[c,b]=1, \alpha c = c^{-1}\alpha, \beta a = a^{-1}\beta c^{k}, \beta b=b^{ -1}\beta c^{-k},\\
    & \beta^2=c,\alpha^2=a, \alpha b=b^{-1}\alpha c^{-k}, \beta \alpha=a^{-1}b^{-1}\alpha\beta c^{-k-1}\rangle
\end{align*}

\begin{lemma}\label{lem:structure:pg}
    Let $k$ be a positive even integer. Then 
    \begin{enumerate}
        \item $\Gamma_k$ fits in the following short exact sequence
    \[1\to \langle a,c \rangle \to \Gamma_k \to \langle \bar \alpha \rangle\ast \langle \bar{\beta} \rangle \to 1\]
    In particular we have a splitting $\Gamma_k= K_\alpha \ast_A K_{\beta}$, where $K_\alpha = \langle c, \alpha \rangle$, $K_{\beta} = \langle c, \beta \rangle$, and $A=\langle a,c \rangle$,

%    \item the subgroup of $\Gamma_k$ generated by $a,b,c,$ and $\alpha$ is a normal subgroup of index 2 isomorphic to $E_k$,

    \item define $\gamma= \alpha\beta$, then $\gamma^2=b$ and $\Gamma_k'$ fits in the following short exact sequence 
     \[1\to \langle c,b \rangle \to \Gamma_k' \to \langle \bar {\gamma a} \rangle\ast \langle \bar{\gamma} \rangle \to 1\]
     Moreover $\Gamma_k'$ is a group of Type pg.
    \end{enumerate}
%    Moreover, following \Cref{notation:sol}, $\Gamma_k'$ is isomorphic to $A\rtimes_\varphi \dbZ$ where $\varphi$ is the automorphism of $A$ given by the matrix $\begin{pmatrix}
%    -1&k\\
%    0&-1
%    \end{pmatrix}$.
\end{lemma}

\begin{proof}
    \begin{enumerate}
        \item The relations that define $\Gamma_k$ imply that the subgroup $\langle a,c\rangle$ is stable under conjugation for $b$, $\alpha$ and $\beta$ therefore $\langle a,c\rangle$ is a normal subgroup of $\Gamma_k$. The relations in the quotient group $p(\Gamma_k)$ that are not trivial are $\bar{\beta} \bar{b}=\bar{b}^{-1}\bar{\beta}$, $\bar{\beta}^2=1$, $\bar{\alpha}^2=1$, $\bar{\alpha}\bar{b}=\bar{b}^{-1}\bar{\alpha}$ and $\bar{\beta}\bar{\alpha}=\bar{b}^{-1}\bar{\alpha}\bar{\beta}$. The last one implies that $\bar{b}=\bar{\alpha}\bar{\beta}\bar{\alpha}\bar{\beta}$, all of these equations imply that $p(\Gamma_k)$ admits the presentation $\langle \bar{\alpha},\bar{\beta}|\bar{\alpha}^2=1,\bar{\beta}^2=1\rangle$, thus $p(\Gamma_k)= \langle \bar{\alpha}\rangle\ast\langle \bar{\beta}\rangle\cong D_\infty$. Since $\alpha^2=a$ and $\alpha c \alpha^{-1}=c^{-1}$, $A$ is a index 2 subgroup of $\langle c, \alpha\rangle$ and this group isn't abelian, so $\langle c,\alpha\rangle=K_\alpha$ is the fundamental group of a Klein bottle, in a similar way $\langle c, \beta\rangle=K_{\beta}$ is the fundamental group of a Klein bottle.

        %Moreover $p(\Gamma_k)$ admits the following presentation $\langle \bar{\alpha},\bar{\beta}|\bar{\alpha}^2=1,\bar{\beta}^2=1\rangle$, then $p(\Gamma_k)\cong D_\infty$.
        \item Following the presentation, we have
        \begin{align*}
            \gamma^2=\alpha\beta\alpha\beta &=\alpha(a^{-1}b^{-1}\alpha\beta c^{-k-1})\beta\\
            &=\alpha(a^{-1}b^{-1}c^{k+1}\alpha\beta )\beta\\
            &=\alpha a^{-1}b^{-1}c^{k+1}\alpha c\\
            &=\alpha a^{-1}b^{-1}c^{k}\alpha\\
            &=a^{-1}\alpha(b^{-1}\alpha)c^{-k}\\
            &=a^{-1}\alpha(\alpha b c^k)c^{-k}=b
        \end{align*}
Since $\Gamma_k'=\langle a,c,\alpha\beta=\gamma\rangle$, we have $aba^{-1}=c^{2k}b$, $cbc^{-1}=b$, $\gamma b\gamma^{-1}=\gamma\gamma^2\gamma^{-1}=\gamma^2=b$, $aca^{-1}=c$, $ccc^{-1}=c$ and 
$\gamma c\gamma^{-1}=\alpha\beta\beta^2\beta^{-1}\alpha^{-1}=\alpha\beta^2\alpha^{-1}=\alpha c\alpha^{-1}=c^{-1}$.
These equations imply that $\langle b,c\rangle$ is normal in $\Gamma_k'$. The group $\Gamma_k'/\langle c\rangle$ admits the presentation $\langle \bar{a}, \bar{\gamma}|\bar{\gamma}\bar{a}=\bar{a}^{-1}\bar\gamma\rangle$ that is isomorphic to $\langle\bar{a}\rangle\rtimes_{-1}\langle \bar{\gamma}\rangle$. With this equivalence we have that $\Gamma_k'/\langle b=\gamma^2,c\rangle=\langle \bar{a}\rangle\rtimes_{-1}\langle \bar{\gamma}|\gamma^2\rangle$ that is isomorphic to $\langle \bar{\gamma a}\rangle\ast\langle \bar{\gamma}\rangle$. Since $\gamma^2=b$ and $\gamma c \gamma^{-1}=c^{-1}$, $\langle c,\gamma \rangle$ is a index 2 subgroup of $\langle b, c \rangle$ and this group isn't abelian, so $\langle b,\gamma\rangle=K_\gamma$ is the fundamental group of a Klein bottle, in a similar way $\langle c, \gamma a\rangle=K_{\gamma a}$ is the fundamental group of a Klein bottle. The automorphism of this amalgamation it's given by $   aca^{-1}=c$, and $a b a^{-1}=c^{2k}b$, equivalently the automorphism is given by the matrix $\begin{pmatrix}
    1&2k\\
    0&1
    \end{pmatrix}$, then $\Gamma_k'$ is a group of pg's case.
        
    \end{enumerate}
\end{proof}

In the following lemma we describe the maximal abelian subgroups of $\Gamma_k$ and their intersections, all the conclusions in this lemma are summarized in \Cref{fig:pgg}.

\begin{lemma}\label{lem:ab:subgps:pgg}
Following \Cref{notation:sol} we have that:
\begin{enumerate}
    \item every abelian subgroup of $\Gamma_k$ is free of rank at most 2,
    
    \item as recorded in the bold lines of \Cref{fig:pgg}, the diagram of subgroups of Type II (which are subgroups of the Type pg group $\Gamma_k'$) and their intersections is as described in \Cref{lem:ab:subgps:pg} and \Cref{fig:pg}, except that the generators $a$ and $b$ are interchanged,
    
    \item  the center of $K_\alpha^{(\beta\alpha)^n}$ (resp. $K_{\beta}^{(\beta\alpha)^n}$) is $\langle ac^{-nk} \rangle$ (resp. $\langle c \rangle$),

    \item let $i,j \in \{ \alpha, \beta \}$, if $K_i^{(\beta\alpha)^n}\neq K_j^{(\beta\alpha)^m}$, then $Z(K_i^{(\beta\alpha)^n})\cap Z(K_j^{(\beta\alpha)^m})$ is trivial,

    \item the intersection of any two subgroups of Type I that lie in a single conjugate $K_i^{(\beta\alpha)^n}$ is $Z(K_i^{(\beta\alpha)^n})$,

    \item the intersection of two subgroups of Type I is trivial if  they belong to distinct conjugates of the factors of $\Gamma_k$,
    
    \item the intersection of $A$ with a subgroup of Type I that is subconjugate to $K_i^{(\beta\alpha)^n}$ is $Z(K_i^{(\beta\alpha)^n})$,

    \item the intersection of $A$ with a subgroup of Type II is either trivial or equal to  $\langle c \rangle$.

    \item the intersection of any subgroup of Type I with any subgroup of Type II is either trivial or equal to  $\langle c \rangle$.
\end{enumerate}
\end{lemma}
\begin{proof}
\begin{enumerate}
    \item Since $\Gamma_k$ is a torsion free virtually poly-$\dbZ$ group of rank 3, any abelian subgroup is free of rank at most 3. On the other hand $\Gamma_k$ cannot have a $\dbZ^3$ subgroup as this posibility only happens, in the context of 3-manifolds, in the presence of an euclidean metric.
    \item Since every Type II subgroup of $\Gamma_k$ is actually a subgroup of $\Gamma_k'$, this claim is a direct consequence of \Cref{lem:structure:pg} (2) and \Cref{lem:ab:subgps:pg}. See also \Cref{fig:pg}.

\item By \Cref{lemma:conjugate:factors} (1) all conjugates of $K_\alpha$ (resp. $K_{\beta}$) are of the form $K_\alpha^{(\alpha\beta)^n}$ (resp. $K_{\beta}^{(\alpha\beta)^n}$) for some $n\in \dbZ$. It is easy to check that $(\alpha^2)^{\alpha\beta} = a^{\alpha\beta}=a^{-1}c^k$ and $(\beta^2)^{\alpha\beta} =c^{\alpha\beta}=c^{-1}$ directly from the presentation. The claim now follows from \cref{lemma:conjugate:factors}(2).

\item By the previous item, the centers $Z(K_i^{(\alpha\beta)^m})$ and $Z(K_j^{(\alpha\beta)^n})$ are of the form $\langle ac^{mk} \rangle$ and $\langle c \rangle$. Thus $Z(K_i^{{(\alpha\beta)}^m})\cap Z(K_j^{{(\alpha\beta)}^n})=1$ provided $Z(K_{i}^{{(\alpha\beta)}^m})\neq  Z(K_{i}^{{(\alpha\beta)}^n})$ (we can think these groups as lines in the lattice plane $\dbZ^2 \cong \langle c, a \rangle$ with diferent slopes). 
\item This is (3) in \Cref{lemma:mabco:klein:bottle}.
\item Let $B$ y $D$ be Type I subgroups that lie in different conjugates $K_i^{(\alpha\beta)^n}$ and $K_j^{(\alpha\beta)^m}$, respectively. Then by \Cref{lemma:mabco:klein:bottle}  (4), $Z(K_i^{(\alpha\beta)^n})$ and $Z(K_j^{(\alpha\beta)^m})$ are index 2 subgroups of $C$ and $D$, respectively. By part (4), the centers have trivial intersection, and thus so do $C$ and $D$.

\item This is (4) in \Cref{lemma:mabco:klein:bottle}.

\item This claim can be read from \Cref{fig:pgg}.

\item Let $B$ be a Type I subgroup and $C$ a Type II subgroup, this means that $p(B)\cong\dbZ/2$ and $p(C)\cong \dbZ$. Then $p(B\cap C)\subset p(B)\cap p(C)=1$, thus $B\cap C\subset A=\langle c,a\rangle$. Therefore $B\cap C=B\cap C\cap A=(B\cap A)\cap(C\cap A)$, and the latter intersection is either trivial or $\langle c \rangle$ by the two previous items.
\end{enumerate}
\end{proof}

\begin{figure}[ht]
    \centering
    \includegraphics[scale=.8]{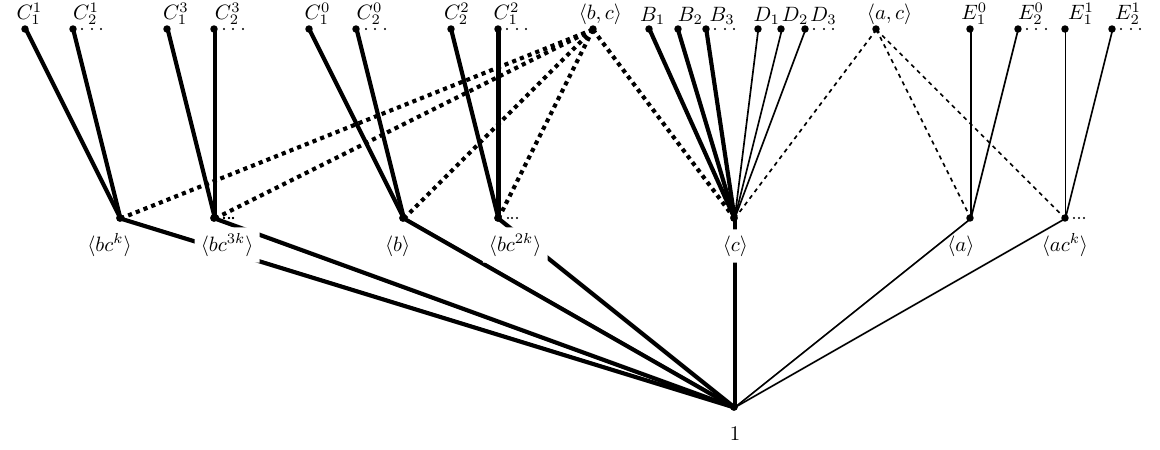}
    \caption{\small The Hasse diagram of  $\mAb{\Gamma_k}$. The subdiagram in bold lines correspond to the Hasse diagram of $\mAb{\Gamma_k'}$ and therefore the groups on the top level are of Type II; note that since $\mAb{\Gamma_k'}$ belongs to the Case pg the subdiagram is identical to \Cref{fig:pg}. Subgroups of the form $D_j$ and $E^i_j$ are of Type I.}
    \label{fig:pgg}
\end{figure}

\begin{corollary}\label{Cor:chains:pgg} Following \Cref{notation:sol} and \Cref{lem:ab:subgps:pg}, every maximal chain in $\mAb{G}$ is of one of the following types:
    \begin{enumerate}
        \item $1 < \langle c \rangle< B$ where $B$ is a subgroup of Type II within $\Gamma_k'$,
        \item $1 < \langle c \rangle< \langle c,b \rangle$,
        \item $1 < \langle c^{mk}b \rangle< C$  where $m\in \dbZ$  and $C$ is an infinite cyclic group that contains $\langle c^{mk} b\rangle$ as a subgroup of index 2,
        \item $1 < \langle c^{mk}b \rangle< \langle c,b \rangle$  where $m\in \dbZ$,

         \item $1 < \langle c \rangle< D$ where $D$ is a subgroup of Type I,
        \item $1 < \langle c \rangle< A$ where $A=\langle a,c \rangle$,
        \item $1 < \langle c^{mk}a \rangle< E$  where $m\in \dbZ$  and $E$ is an infinite cyclic group that contains $\langle c^{mk}a \rangle$ as a subgroup of index 2,
        \item $1 < \langle c^{mk}a \rangle<\langle a,c\rangle$  where $m\in \dbZ$,
    \end{enumerate}
\end{corollary}

\begin{theorem}\label{thm:case:pgg} Let $k$ be a positive integer, then $\Ecom{E_k}\simeq \bigvee_{\mathbb{N}} S^1$.
\end{theorem}

\begin{proof} This proof is very similar to that of \Cref{Cor:chains:pg}. Any maximal chain in $\mAbCo{\Gamma_k}$ is obtained by multiplying one of the chains listed in \Cref{Cor:chains:pgg} by a group element $g$.
For a given choice of $g$, the edge going from $\{g\}$ to the maximal element of a chain in cases (1), (3), (5) or (7) is contained only in that triangle. Thus we can remove those edges and triangles without the changing the homotopy type. After removing those, in each maximal chain of type (2), (4), (6) or (8), the edge from $\{g\}$ to the middle element is contained in a unique triangle, so we can now removes those edges and triangles. After this, no triangles remain, showing $\mAbCo{\Gamma_k}$ has the homotopy type of a graph, and thus of a wedge of circles (since it is connected).
The complex obtained after removing the triangles still contains a copy of $\mAbCo{\Gamma_k'}$, which has the homotopy type of a countable wedge of circles, and therefore so does $\mAbCo{\Gamma_k}$.
\end{proof}

\section{3-manifolds modeled on $\sol$}

By \cite[Theorem~1.8.2]{AFW15}, the fundamental groups of manifolds modeled on $\sol$ can be split into two categories that we described below.

\begin{description}
\item[Semi direct products] Groups of the form $\dbZ^2\rtimes_\varphi \dbZ$ where $\varphi$ is an automorphism of $\dbZ^2$ that does not fix any subgroup of rank 1.
\item[Amalgams] Groups of the form $K\ast_A K$ where where $K$ is the fundamental group of the Klein bottle,  $A=\dbZ \times 2\dbZ\cong \dbZ^2$, and the amalgamation is given by an automorphism $\varphi\colon A\to A$  that does not fix any subgroup of rank 1.
\end{description}

\begin{theorem}\label{prop:sol:semidirect}
Let $G=\dbZ^2\rtimes_\varphi \dbZ$ where $\varphi$ is an automorphism of $\dbZ^2$ that does not fix any subgroup of rank 1 . Then 
\begin{enumerate}
    \item every abelian subgroup of $G$ is free of rank at most 2,
    \item the subgroup $\dbZ^2\rtimes \{0\}$ is the unique maximal abelian subgroup of rank 2,
    \item the intersection of any two maximal abelian subgroups of $G$ is trivial, and
    \item $\Ecom{G}\simeq \bigvee_{\mathbb{N}} S^1$.
\end{enumerate}
\end{theorem}
\begin{proof}
First we prove (1). Since $G$ is a finitely generated and torsion-free poly-$\dbZ$ group of rank 3, every abelian subgroup of $G$ is of the form $\dbZ^n$ with $n\leq \mathrm{rank}(G)=3$. Notice that $G$ cannot have a subgroup isomorphic to $\dbZ^3$, because it would have finite index by \cref{prop:polyZ}, but the only 3-manifold groups with this property are the ones modeled on $\mathbb{E}^3$.

Item (2) follows from \cite[Lemma~4.8]{LASS21}. Since item (4) follows directly from item (3) and \Cref{thm:countable:wedge}, the rest of the proof deals with the proof of item (3).

Denote $H=\dbZ^2\rtimes \{0\}$. Let $C$ be a maximal abelian subgroup of $G$ different from $H$. We claim that $C\cap H$ is trivial. Let $p\colon G \to \dbZ$ be the projection onto the second factor of $G$. If $C\cap H$ is not trivial, then it must be a infinite cyclic subgroup of the kernel of $p$, and therefore $p(C\cap H)$ has rank zero, that is, $p(C\cap H)$ is finite, and thus it is trivial. Hence $C\cap H$ is contained in $H$ which contradicts the maximality of $C$, this proves the claim.

Let us pause our proof to show that $\mathrm{comm}_GC=C$, where the commensurator is given by $\mathrm{comm}_GC=\{g\in G | gCg^{-1}\cap C\neq 1\}$. As $\mathrm{comm}_GC$ is a subgroup of $G$, it is poly-$\dbZ$ by \Cref{prop:polyZ}. Assume  $\mathrm{comm}_GC$ contains a subgroup $L$ of rank 2, then $L\cong \dbZ\rtimes \dbZ$ and therefore it contains a subgroup isomorphic to $\dbZ^2$, which is contained in $H$ for being the unique maximal $\dbZ^2$ subgroup of $G$. Since $\mathrm{comm}_GC$ contains both $L$ and $C$, and $C\cap L$ is trivial, we have that $\mathrm{comm}_GC$ is of rank 3, and therefore has finite index in $G$ by \Cref{prop:polyZ}. This implies that $\mathrm{comm}_GC$ also is of the form $\dbZ^2\rtimes_\psi \dbZ$ where $\psi$ is an automorphism of $\dbZ^2$ that does not fix any subgroup of rank 1, in particular  this group does not have any nontrivial normal cyclic subgroups by \cite[Lemma~4.8(b)]{LASS21}. On the other hand, $\mathrm{comm}_GC$ normalizes a finite index subgroup of $C$ (see for instance \cite[Lemma~5.15]{LW12}), this leads to a contradiction. Thus we conclude $\mathrm{comm}_GC$ is of rank one, that is $\mathrm{comm}_GC$ is infinite cyclic. Since $C\subseteq \mathrm{comm}_GC$, by maximality of $C$ we have  $C=\mathrm{comm}_GC$.

Let $C$ and $D$ be two maximal abelian subgroups of $G$ isomorphic to $\dbZ$, such that $C\cap D$ is nontrivial. Thus $C$ and $D$ are commensurable, and therefore $\mathrm{comm_G}C=\mathrm{comm_G}D$. By the previous paragraph, we conclude $C=D$. This finishes the proof.
\end{proof}

In the following lemma we describe the maximal abelian subgroups of $K\ast_A K$ and their intersections, all the conclusions in this lemma are summarized in \Cref{fig:sol}.

\begin{lemma}\label{lem:ab:subgps:solamalgam}
Following \Cref{notation:sol} we have that:
\begin{enumerate}
    \item every abelian subgroup of $G$ is free of rank at most 2,
    \item the subgroup $A$ is the only rank 2 maximal abelian subgroup of $G$, 
    \item the intersection of any two subgroups of Type II is trivial, and the intersection of any subgroup of Type II with $A$ is trivial,
    \item The intersection of any subgroup of Type I with any subgroup of Type II is trivial,

    \item the intersection of any two subgroups of Type I that lie in a single conjugate $K_i^x$ is $Z(K_i^x)$,
    \item the intersection of $A$ with a subgroup of Type I that is subconjugate to $K_i^x$ is $Z(K_i^x)$,
    \item if $K_i^x\neq K_j^y$, then  $Z(K_i^x)\cap Z(K_j^y)$ is trivial.
\end{enumerate}
\end{lemma}
\begin{proof}
\begin{enumerate}
    \item Since $G$ is a torsion free poly-$\dbZ$ group of rank 3, any abelian subgroup is free of rank at most 3. On the other hand $G$ cannot have a $\dbZ^3$ subgroup as this posibility only happens, in the context of 3-manifolds, in the presence of an euclidean metric.

    \item Let $H$ be a rank 2 abelian subgroup of $G$. Note that $p(H)$ cannot be an infinite subgroup of $\dbZ\rtimes_{-1} \dbZ/2$. Indeed, if $p(H)$ were infinite it would mean that a generator of $p(H)$ stabilizes the rank one subgroup $H\cap A$, which is impossible. Then $H'=H\cap G'$ is a rank 2 abelian subgroup subconjugate to one of the factors of $G$, and we conclude from \Cref{lemma:mabco:klein:bottle} that $H$ is a subgroup of $A$.
 
    \item It follows directly from \Cref{prop:sol:semidirect}(3).

    \item Let $H_1$ and $H_2$ be subgroups of Type I and II respectively. Then $p(H_1)$ is finite and $p(H_2)$ is infinite cyclic, so $p(H_1) \cap p(H_2) = 1$. Therefore $H_1 \cap H_2 \subseteq \ker p = A$. But by the previous part, $H_2 \cap A = 1$, and hence $H_1 \cap H_2 = 1$ too.
    
    \item It follows directly from \Cref{lemma:mabco:klein:bottle}(4).

    \item It follows directly from \Cref{lemma:mabco:klein:bottle}(5).

    \item  Note that $Z(K_i^x)\cap Z(K_j^y)=1$ is equivalent to $Z(K_i)\cap Z(K_j^{x^{-1}y})=1$. Then, let us prove $Z(K_i)\cap Z(K_j^{x})=1$, provided $K_i\neq K_j^{x}$. There are two cases: 

    \begin{description}
        \item[Case 1. $i=j$] Let us first handle the case in which $x \in G'$. 
        Since $G'$ is isomorphic to $\dbZ^2\rtimes_{\varphi'}\dbZ$, we have that for  $a\in A$ (in particular for $a\in Z(K_i^x)$), $xax^{-1}=\varphi'^{(n)}(a)$ for some $n\in\dbZ$. Since $K_i \neq K_i^x$, we have $n \neq 0$. Then
        \[Z(K_i)\cap Z(K_i^{x})=Z(K_i)\cap \varphi'^{(n)}(Z(K_i))=1.\]
        where the last equality holds because $\varphi'$ doesn't fix any rank 1 subgroup of $A$.
        Now, we handle the general case. Since $G'$ has index 2, then $x^2\in G'$. By the previous argument, 
        $Z(K_i)\cap Z(K_i^{x^2})=1$. If we had $Z(K_i)\cap Z(K_i^x) \neq 1$, let $a$ be a generator of that intersection.
        We have $xax^{-1} = a^{\pm 1}$, and thus $x^2 a x^{-2} = a$, contradicting that $Z(K_i)\cap Z(K_i^{x^2})=1$.
        \item[Case 2. $i\neq j$]. Note that $H=\langle K_1,K_2^x\rangle$ has finite index in $G$. Hence the manifold $\tilde{M}/H$ is a 3-manifold modeled on $\sol$, where $\tilde{M}$ is the universal cover of the initial 3-manifold. Hence, $H$ is isomorphic to $K_1\ast_A K_2$  by an automorphism $\psi\colon A\to A$ that does not fix any subgroup of rank 1. Since $Z(K_1)\cap Z(K_2^x)\subset Z(H)=1$, the result follows.
    \end{description}
\end{enumerate}
\end{proof}

\begin{figure}[ht]
    \centering
    \includegraphics[scale=.8]{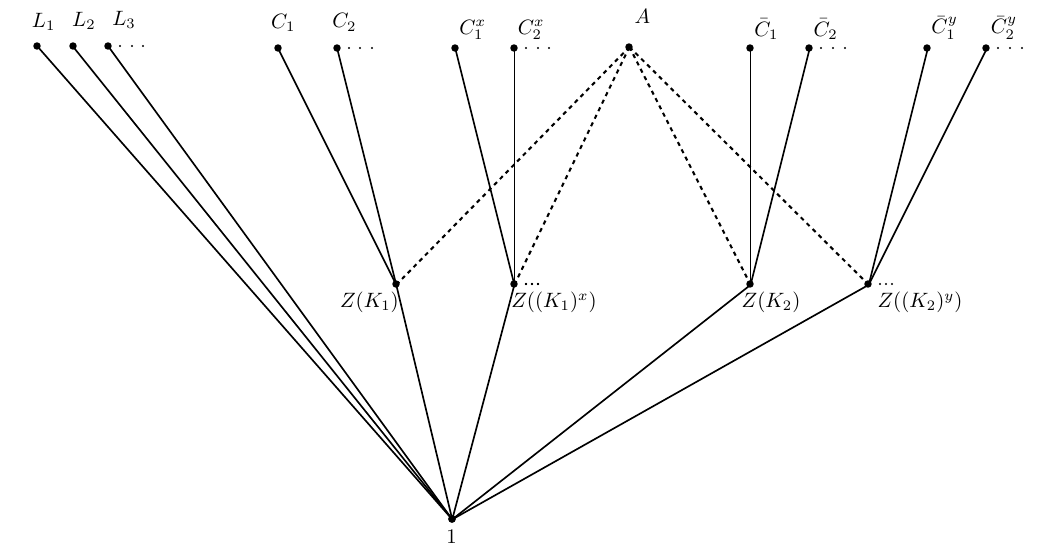}
    \caption{\small Hasse diagram of $\mAb{K_1\ast_A K_2}$. The subgroups of the form $L_i$ are of Type II. The subgroups of the form $C_i$, $C_i^x$, $\bar C_i$ and $\bar C_i^y$ are of Type I.}
    \label{fig:sol}
\end{figure}

The following corollary is an immediate consequence of the preceding lemma.

\begin{corollary}\label{Cor:chains:sol} Following \Cref{notation:sol} and \Cref{lem:ab:subgps:solamalgam}, every maximal chain in $\mAb{G}$ is of one of the following types:
    \begin{enumerate}
        \item $1 < L$ with $L$ of Type II,
        \item $ 1< Z(K_i^x) < A$ for some $x\in G$ and $i\in \{1,2\}$,
        \item $ 1< Z(K_i^x) < C$ for some $x\in G$ and $i\in \{1,2\}$, and $C$ of Type I.
    \end{enumerate}
\end{corollary}

\begin{theorem}\label{thm:Ecom:Sol:amalgam}
Let $G$ be a group of the form  form $K\ast_A K$ where $A=\dbZ \times 2\dbZ\cong \dbZ^2$ and the amalgamation is given by an automorphism $\varphi\colon A\to A$  that does not fix any subgroup of rank 1. Then, $\Ecom{G}\simeq \bigvee_{\mathbb{N}} S^1$.
\end{theorem}

\begin{proof}
As a consequence of \Cref{Cor:chains:sol}, all maximal chains in $\mAbCo{G}$ are of one of the following types:
    \begin{enumerate}
        \item $\{g\} < gL$ with $L$ of Type II,
        \item $\{g\} < gZ(K_i^x) < gA$ for some $x\in G$ and $i\in \{1,2\}$,
        \item $\{g\} < gZ(K_i^x) < gC$ for some $x\in G$ and $i\in \{1,2\}$, and $C$ of Type I.
    \end{enumerate}
Note that for a given choice of $g \in G$ and subgroup $C$ of Type I, there is a unique triangle of the form (3) in the above list, so we can remove the edge ${g} < gC$ and the interior of the triangle without changing the homotopy type. After removing those triangles, each edge of the form $\{g\} < g Z(K_i^x)$ is in a unique triangle, which is of the form (2). We can remove those edges and the interior of their containing triangles without changing the homotopy type either.

At this point, we have no triangles remaining, which shows $\mAbCo{G}$ has the homotopy type of a graph, and hence of a wedge of circles (since it is connected). The complex obtained after removing the triangles still contains a copy of $\mAbCo{G'}$ (comprising the edges of the form (1)), which has the homotopy type of a countable wedge of circles, and therefore so does $\mAbCo{G}$.
\end{proof}

\appendix

\section{Computations for the spherical case using GAP}\label{Appendix:code:GAP}

We will exhibit GAP code that can be used to verify that $\mAb{G}$, and therefore also $\mAbCo{G}$, is of height 1.

First, we define a function to compute the maximal sets under inclusion in a given family of sets. A set $s$ in the family is maximal if the list of sets containing $s$ has length one (that is, consists solely of $s$ itself):

\begin{lstlisting}[language=GAP]
maximal:=sets->Filtered(sets,
                s->Length(Filtered(sets,t->IsSubset(t,s)))=1);
\end{lstlisting}

Next we define functions to compute the maximal abelian subgroups of a group $G$, and to compute the family of pairwise intersections of maximal abelian subgroups:

\begin{lstlisting}[language=GAP]
maxAbSub := G->maximal(Filtered(AllSubgroups(G), IsAbelian));

intMaxAbSub := G -> List(Combinations(maxAbSub(G),2),
                     P->Intersection(P[1],P[2]));
\end{lstlisting}

To check whether a group $G$ has $\mAbCo{G}$ of height 1, we simply check if all pairwise intersections of maximal abelian subgroups are equal to the center of $G$:

\begin{lstlisting}[language=GAP]
isHeight1 := G -> ForAll(intMaxAbSub(G), A -> A = Center(G));
\end{lstlisting}

When a group $G$ has $\mAbCo{G}$ of height 1, the geometric realization of $\mAbCo{G}$ is a graph, and thus homotopy equivalent to a wedge of circles. We can compute how many circles by computing the Euler characteristic of the graph:

\begin{lstlisting}[language=GAP]
circles := function(G)
  local maxAb, Z, edges, vertices;
  Assert(0, isHeight1(G));
  maxAb := maxAbSub(G);
  Z := Center(G);
  edges := Index(G,Z)*Length(maxAb);
  vertices := Sum(maxAb,A->Index(G,A)) + Index(G,Z); 
  return edges - vertices + 1;
end;
\end{lstlisting}

\smallskip

Now we can easily compute the homotopy type of $\Ecom{P_{48}}$:

\begin{lstlisting}[language=GAP]
gap> F := FreeGroup("x", "y");; x := F.1;; y := F.2;;
gap> G := F / [x^(-2) * (x*y)^3, x^(-2) * y^4, x^4];; IdGroup(G);
[ 48, 28 ]
gap> isHeight1(G);
true
gap> circles(G);
167
\end{lstlisting}

\smallskip

Similarly, we can compute the homotopy type of $\Ecom{P_{120}}$:

\begin{lstlisting}
gap> G := F / [x^(-2) * (x*y)^3, x^(-2) * y^5, x^4];; IdGroup(G);
[ 120, 5 ]
gap> isHeight1(G);
true
gap> circles(G);
1079
\end{lstlisting}

\smallskip

Finally, we compute the number of circles in  $\Ecom{P'_{8\cdot 3^m}}$. This computation is more involved since we are not talking about a single group but rather about a family of groups, one for each $m \ge 1$. The proof of \Cref{thm:Ecom:p'} outlines the computation leaving only a couple of claims to check with GAP, which we now do.

The first claim is that $x$ and $y$ do not commute in $P'_{8\cdot 3^m}$. To verify this, note that $P'_{24} = P'_{8\cdot 3^1}$ is a quotient of $P'_{8\cdot 3^m}$ for any $m \ge 1$, so it suffices to check $x$ and $y$ do not commute in $P'_{24}$:

\begin{lstlisting}
gap> F := FreeGroup("x", "y", "z");;
gap> x := F.1;; y := F.2;; z:=F.3;;
gap> G := F / [x^2 * y^(-2), x^2 * (x*y)^(-2), z^3,
               (x^z) * y^(-1), (y^z) * (x*y)^(-1)];;
gap> IdGroup(G);
[ 24, 3 ]
gap> IsAbelian(Subgroup(G, [G.1, G.2]))
false
\end{lstlisting}

Having verified that claim, the proof of \Cref{thm:Ecom:p'} show that all $\Ecom{P'_{8\cdot 3^m}}$ are a wedge of the same number of circles, so to verify the number of circles claimed there it suffices to check $\Ecom{P'_{24}}$:

\begin{lstlisting}
gap> isHeight1(G);
true
gap> circles(G);
39
\end{lstlisting}

\AtNextBibliography{\small}
\printbibliography
% \bibliographystyle{alpha} %harvard, unsrt, alpha
% \bibliography{ref}
\end{document}